\documentclass[a4paper,reqno]{amsart} 

\usepackage{amssymb}
\usepackage[mathcal]{euscript} 

\usepackage{tikz-cd} 
\usetikzlibrary{babel}

\usepackage{hyperref}

\newcommand{\bydef}{:=}

\newcommand{\id}{\mathrm{id}}%identity map
\newcommand{\proj}{\mathrm{pr}}%projection map
\DeclareMathOperator{\rank}{\mathrm{rank}} %rank

\newcommand{\bi}{\mathbf{i}}%imaginary unit

\newcommand{\cA}{\mathcal{A}}%algebras

\newcommand{\cH}{\mathcal{H}}

\newcommand{\cL}{\mathcal{L}}

\newcommand{\cT}{\mathcal{T}}

\newcommand{\frg}{{\mathfrak g}}%Lie algebras

\newcommand{\ZZ}{\mathbb{Z}}

\newcommand{\FF}{\mathbb{F}} 

\DeclareMathOperator{\rad}{\mathrm{rad}}%nilradical
\DeclareMathOperator{\Hom}{\mathrm{Hom}}
\DeclareMathOperator{\im}{\mathrm{im}\,}
\DeclareMathOperator{\Aut}{\mathrm{Aut}}%automorphism group
\DeclareMathOperator{\Stab}{\mathrm{Stab}}%stabilizer
\DeclareMathOperator{\Diag}{\mathrm{Diag}}
\DeclareMathOperator{\supp}{\mathrm{Supp}}

\newcommand{\frsl}{{\mathfrak{sl}}}

\newcommand{\GL}{\mathrm{GL}} 
\newcommand{\SL}{\mathrm{SL}}

\newcommand{\Ort}{\mathrm{O}}

\newcommand{\SP}{\mathrm{Sp}}

\DeclareFontFamily{U}{matha}{\hyphenchar\font45}
\DeclareFontShape{U}{matha}{m}{n}{
      <5> <6> <7> <8> <9> <10> gen * matha
      <10.95> matha10 <12> <14.4> <17.28> <20.74> <24.88> matha12
      }{}
\DeclareSymbolFont{matha}{U}{matha}{m}{n}
\DeclareFontFamily{U}{mathx}{\hyphenchar\font45}
\DeclareFontShape{U}{mathx}{m}{n}{
      <5> <6> <7> <8> <9> <10>
      <10.95> <12> <14.4> <17.28> <20.74> <24.88>
      mathx10
      }{}
\DeclareSymbolFont{mathx}{U}{mathx}{m}{n}

\DeclareMathSymbol{\obot}         {2}{matha}{"6B}
\DeclareMathSymbol{\bigobot}       {1}{mathx}{"CB}

\newenvironment{romanenumerate} 
{\begin{enumerate}

}{\end{enumerate}}

\newcommand{\buno}{\mathbf{1}}

\newcommand{\oR}{\overline{R}}
\newcommand{\oE}{\overline{E}}
\newcommand{\oF}{\overline{F}}
\newcommand{\oW}{\overline{W}}
\newcommand{\oG}{\overline{G}}
\newcommand{\oH}{\overline{H}}
\newcommand{\oT}{\overline{T}}
\newcommand{\oGamma}{\overline{\Gamma}}
\newcommand{\opi}{\overline{\pi}}

\DeclareMathOperator{\Arf}{\mathrm{Arf}}
\newcommand{\bcero}{\mathbf{0}}
\newcommand{\bc}{\mathbf{c}}
\newcommand{\cF}{\mathcal{F}}
\DeclareMathOperator{\Inv}{\mathrm{Inv}}

\newtheorem{theorem}{Theorem}[section]
\newtheorem{proposition}[theorem]{Proposition}
\newtheorem{lemma}[theorem]{Lemma}
\newtheorem{corollary}[theorem]{Corollary}

\theoremstyle{definition} 
\newtheorem{definition}[theorem]{Definition}
\newtheorem{example}[theorem]{Example}

\theoremstyle{remark} \newtheorem{remark}[theorem]{Remark}
\numberwithin{equation}{section}

\def\bigstrut{\vrule height 13pt width 0ptdepth 5pt}

\begin{document}

\title%[]%
{Special pure gradings on simple Lie algebras of types $E_6$, $E_7$, $E_8$}

\author[C.~Draper]{Cristina Draper}
\address[C.D.]{Departamento de \'Algebra, Geometr{\'\i}a y Topolog{\'\i}a,
Universidad de M\'alaga, 29071 M\'alaga, Spain}
\email{cdf@uma.es}
\thanks{C.D. has been supported  by the Spanish Ministerio de Ciencia e Innovaci\'on   through 
project  PID2023-152673NB-I00 and by the Junta de Andaluc\'{\i}a  through project  PPRO-FQM336-G-2023, 
 all of them with FEDER funds }

\author[A.~Elduque]{Alberto Elduque} 
\address[A.E.]{Departamento de
Matem\'{a}ticas e Instituto Universitario de Matem\'aticas y
Aplicaciones, Universidad de Zaragoza, 50009 Zaragoza, Spain}
\email{elduque@unizar.es} 
\thanks{A.E. has been supported by grant
PID2021-123461NB-C21, funded by 
MCIN/AEI/10.13039/ 501100011033 and by
 ``ERDF A way of making Europe''; and by grant 
E22\_20R (Gobierno de Arag\'on).}

\author[M.~Kochetov]{Mikhail Kochetov}
\address[M.K.]{Department of Mathematics and Statistics, 
	Memorial University of Newfoundland, St.~John's, NL, A1C5S7, Canada}
\email{mikhail@mun.ca}
\thanks{M.K. has been supported by Discovery Grant 2018-04883 of the Natural Sciences and Engineering Research Council (NSERC) of Canada.}

\subjclass[2020]{Primary 17B70; Secondary 17B25; 17B40}

\keywords{Graded algebras, exceptional Lie algebras, pure gradings, special gradings, Weyl group, quadratic form, Arf invariant}

%\date{}

\begin{abstract}
A group grading on a semisimple Lie algebra over an 
algebraically closed field of characteristic zero is 
\emph{special} if its identity component is zero; it is 
\emph{pure} if at least one of its components, other than the identity 
component, contains a Cartan subalgebra.
We classify special pure gradings on Lie algebras of types 
$E_6$, $E_7$, $E_8$ up to equivalence and up to isomorphism. To 
this end, we use quadratic forms over the field of two 
elements to show that there are exactly three equivalence 
classes for $E_6$, four for $E_7$, and five for $E_8$. The 
computation of the corresponding Weyl groups
and their actions on the universal groups yields a set of 
invariants that allow us to distinguish the isomorphism classes.
\end{abstract}

\maketitle

\section{Introduction}

Since the emergence of the theory of Lie algebras in the late 
19th and early 20th centuries, gradings have been a key tool 
for studying the algebras themselves as well as their 
representations. 
After the classical root space decomposition of any semisimple 
Lie algebra---which provides a grading over the root lattice 
associated with a Cartan subalgebra---the study of gradings on 
Lie algebras has developed into a rich area of research. Over 
the past several decades, the most frequently occurring 
gradings in the literature have been those by the cyclic groups 
\( \mathbb{Z}_2 \) and \( \mathbb{Z} \):
\( \mathbb{Z}_2 \)-gradings appear prominently in the theory of 
symmetric spaces in Differential Geometry (see the book 
\cite{helgason}),  
while 
\( \mathbb{Z} \)-gradings naturally arise in filtered Lie 
algebras, in the context of Kac-Moody algebras, and in the 
study of the Virasoro and Heisenberg algebras. 
For other infinite-dimensional Lie algebras, gradings by 
$\ZZ^n$ as well as other finitely generated abelian groups help 
in constructing and classifying deformations as well as in the 
study of their representations.

 Patera and Zassenhaus \cite{PaZa} initiated the task of 
 systematically developing  the theory of gradings on Lie 
 algebras from an algebraic perspective. They formulated the 
 problem of determining fine gradings in addition to the root 
 space decomposition. 
 For a group $G$, a $G$-grading on an algebra $\cA$ is a 
 decomposition, $\cA=\oplus_{g\in G}\cA_g$, into direct sum of 
 subspaces, labeled by the elements of $G$ and called 
 \emph{homogeneous components}, such that the product in the 
 algebra agrees with the group operation: 
 $\cA_g\cA_h\subset\cA_{gh}$. A grading is called \emph{fine} if its 
 homogeneous components are as small as possible, i.e., it does 
 not have a proper refinement.
 
 The classification of fine gradings on simple Lie algebras 
 over algebraically closed fields of characteristic zero became 
 a major milestone, ultimately compiled in the monograph 
 \cite{EKmon}. This book presents many classification results for 
 gradings on simple Lie algebras, due to the contributions of 
 many authors (see the references in \cite{EKmon} and, in 
 addition, \cite{YuEs}).
 More precisely, one can consider two types of classification 
 for gradings: up to equivalence and up to isomorphism.
 Two gradings on $\cA$ are said to be \emph{equivalent} (respectively, 
 \emph{isomorphic}) if there exists an automorphism of $\cA$ that maps 
 each component of the first grading onto some component 
 (respectively, the component with the same label) of the 
 second grading.
 The classification up to isomorphism is particularly relevant 
 in applications where the specific labeling of components 
 plays a role, as is the case when the group $G$ carries some 
 additional structure such as the commutation factor in the 
 theory of Lie color algebras.

Since any grading can be obtained by joining some of the 
components of a fine grading through a process known as 
\emph{coarsening}, the classification of fine gradings up to equivalence 
can serve as the starting point for a complete classification of gradings. 
 However, implementing this in practice is far from 
 straightforward, one difficulty being that the same grading 
 may arise as a coarsening of several nonequivalent fine 
 gradings. 
Indeed, the classification of all gradings up to isomorphism 
for classical simple Lie algebras was obtained in a different 
way, namely, by transferring the problem to certain associative 
algebras and then using the structure theory of graded 
associative algebras. The classification of all gradings up to 
isomorphism for the exceptional Lie algebras $G_2$ and $F_4$ 
was derived from fine gradings by ad hoc arguments, but the 
situation is far more complex for $E_6$, $E_7$ and $E_8$ (each having 
$14$ nonequivalent fine gradings), so the problem remains open 
for these types, even over algebraically closed fields of 
characteristic zero. This is a significant problem that should 
be resolved, since the classification of graded-simple finite-dimensional 
Lie algebras would follow from the classification 
of gradings on all simple finite-dimensional Lie algebras via 
the so-called loop construction (see \cite{loopconstruction}). 
It should also be noted that the classification of all gradings 
on classical simple Lie superalgebras (with nonzero odd part) 
is now known \cite{H_thesis,AKY,HK}, so completing the $E$ 
types would also yield the classification of simple Lie color 
algebras via the so-called discoloration of Scheunert.  

Due to the non-uniqueness issue mentioned above, the 
classification of gradings up to isomorphism has not yet been 
attempted for the $E$ types. A new approach was recently 
suggested in \cite{EK_almostfine}, using what the authors 
called \emph{almost fine gradings}. The definition is not 
immediately intuitive, but the idea is that a grading is almost 
fine if it cannot be refined while keeping the same toral rank 
(for instance, any fine grading is almost fine). 
The advantage of almost fine gradings is that any $G$-grading 
can be obtained by coarsening from an almost fine grading in an 
essentially unique way.
To obtain a classification of $G$-gradings using this approach, one 
needs a list of the equivalence classes of almost fine gradings 
as well as their Weyl groups together with the action on their 
universal groups.

The purpose of this paper is to initiate the classification of gradings 
up to isomorphism on the simple Lie algebras of the $E$ types by classifying, 
both up to isomorphism and up to equivalence, the special pure gradings 
on these algebras. 
A group grading on a semisimple Lie algebra $\cL$ over an 
algebraically closed field of characteristic zero is 
\emph{pure} if at least one of its components, other than the 
identity component, contains a Cartan 
subalgebra of $\cL$. This concept was introduced in 
\cite{Hesselink}, motivated by the so-called \emph{Jordan 
gradings} (see \cite{Alek,Jordangradings}). The particular 
symmetry of these  gradings makes them especially useful for 
constructing manageable models of exceptional Lie algebras (for 
instance, the one in \cite{CristinaG2}).
A grading on $\cL$ is \emph{special} if its identity component 
is zero; these are precisely the almost fine gradings on $\cL$ 
that have toral rank $0$. 

Partial information on almost fine gradings and their Weyl 
groups for simple Lie algebras of the $E$ types appears in 
\cite{Yu,YuEs}, where the objectives and methods are quite 
different from ours: the author classifies certain closed 
abelian subgroups of exceptional compact real simple Lie groups 
of adjoint type. The diagonalization of the action of these 
subgroups on the complexified tangent Lie algebra gives the 
almost fine gradings on it. We do not rely on these works, 
but we will point out some connections with our approach.

The structure of the paper is the following. Section~\ref{se_preliminares} contains preliminaries on 
quadratic forms in characteristic $2$, and how some of the orthogonal groups over the field of two 
elements, $\FF_2$, appear in relation with the Weyl groups of the simple Lie algebras of the $E$ types. 
These quadratic forms (as well as those induced by them on certain subquotients) will appear later in 
the explicit computations of the Weyl groups of the special pure gradings on these algebras. 
We also review the main concepts related to gradings (universal groups, Weyl groups, almost 
fine gradings) that are essential for achieving a classification up to isomorphism. 
Section~\ref{se:SpecialPure} reviews the work \cite{Hesselink}, aiming to adapt its results 
to our notation and goals. It also addresses general properties of special pure gradings 
and their Weyl groups---specifically, the subgroup that preserves a homogeneous component 
which is a Cartan subalgebra.
Section~\ref{se:SpecialpureE} is devoted to the classification, up to equivalence, of special pure 
gradings on the simple Lie algebras of the $E$ types, based on the quadratic forms mentioned above.
The main results here are Theorems~\ref{th:E8}, \ref{th:E7}, and \ref{th:E6}. Further information 
about the Weyl groups of the special pure gradings on these algebras is presented in
 Section~\ref{se:WeylSpecialPureE}. 
The goal is to determine what is preserved by these groups, so as to define a set of invariants 
for the gradings in Definition~\ref{df:inv}.
This leads to a classification up to isomorphism, presented in Theorem~\ref{th:iso}.

%%%%%%%%%%%%%%%%%%%%%%%%%%%%%%%%%%%%%%%%%%%%%%%

\bigskip

\section{Preliminaries}\label{se_preliminares}

\subsection{Quadratic forms in characteristic 2 }\label{subse_quadratic}

In this work, we will need to understand the Weyl groups of 
certain gradings on the simple Lie algebras of the $E$-types, 
including their usual Weyl groups, which turn out to be closely 
related to certain orthogonal groups over the field of two 
elements, $\FF_2$. Since quadratic forms in characteristic $2$ 
are perhaps not so well-known, we start with reviewing them 
briefly. The material in this section is standard and has been 
extracted from the textbook \cite{Grove}. 

If $V$ is a (finite-dimensional) vector space over a field 
$\FF$, recall that $Q\colon V\to\FF$ is a \emph{quadratic form} if 
$Q(cx)=c^2Q(x)$ for any $c\in\FF$ and 
$B_Q(x,y):=Q(x+y)-Q(x)-Q(y)$ is a bilinear form, called the 
\emph{polarization} of $Q$. Observe that $B_Q$ is necessarily 
symmetric (i.e., $B_Q(x,y)=B_Q(y,x)$) and, moreover, 
alternating in characteristic $2$ (i.e., $B_Q(x,x)=0$). 
In characteristic $2$, $Q$ cannot be recovered from $B_Q$. 

If $U$ and $V$ are two vector spaces endowed with
quadratic forms $Q_U$ and $Q_V$, their orthogonal sum, denoted by
$U\obot V$, is the direct sum $U\oplus V$, endowed with
the quadratic form $Q(u+v)\bydef Q_U(u)+Q_V(v)$ for all 
$u\in U$ and $v\in V$. Then $U$ and $V$ are orthogonal to each
other with respect to the polarization $B_Q$. We will also  write
$Q=Q_U\obot Q_V$.

A symmetric bilinear form $B$ is said to be 
\emph{nondegenerate} if its radical $\rad B=\{x\in V\mid 
B(x,V)=0\}$ vanishes. A quadratic form $Q$ is nondegenerate if $B_Q$ is nondegenerate.
In characteristic $2$, this concept is extended, and we will use 
the terms \emph{regular} or \emph{nonsingular} quadratic form to avoid 
possible confusion. 
Note that the restriction $Q\vert_{\rad B_Q}$ is semilinear 
(i.e., $Q(x+cy)=Q(x)+c^2Q(y)$ for $x,y\in \rad B_Q$ and 
$c\in\FF$), and the quadratic form is said to be \emph{regular} 
if the radical of $Q$, defined as $\rad Q\bydef\ker(Q\vert_{\rad B_Q})$ is trivial. 
For perfect fields, the regularity forces that $\dim\rad B_Q$ 
is 0 or 1. Moreover, the possibilities for a regular $Q$ are 
not many.
According to \cite[Theorem 12.9]{Grove}, 
\begin{enumerate}
\item If $\dim V=2m+1$, there is a basis $\{v_i\}$ of $V$ such 
that $Q(\sum_ix_iv_i)=x_1x_2+x_3x_4+\dots+x_{2m-1}x_{2m}+x_{2m+1}^2$.

(In this case $ \rad B_Q=\FF v_{2m+1}$ and $Q(v_{2m+1})=1$.)
\item If $\dim V=2m$, there is a basis $\{v_i\}$ of $V$ such that either
\begin{itemize}
\item[\rm(a)] $Q(\sum_ix_iv_i)=\sum_{i=1}^mx_{2i-1}x_{2i}$, or
\item[\rm(b)] $Q(\sum_ix_iv_i)=\sum_{i=1}^mx_{2i-1}x_{2i}+x_{2m-1}^2+cx_{2m}^2$, 
for some $c\in\FF$ such that $x^2+x+c$ is 
irreducible in $\FF[x]$.

(In this case the polarization $B_Q$ is nondegenerate.)
\end{itemize}
\end{enumerate}

In characteristic 2,  a \emph{hyperbolic plane} $H$ is spanned 
by $u,v\in V$ with $B_Q(u,v)=1$ and $Q(u)=Q(v)=0$. 
The maximum $r$ such that $V=H_1\obot H_2\obot \dots \obot H_r\obot W$, 
with each $H_i$ a hyperbolic plane, is called the 
\emph{Witt index}, and coincides with the dimension of a 
maximal totally isotropic subspace of $V$ (in fact, any two of 
these are isometric, due to Witt's Extension Theorem). The Witt 
index equals $m$ in case (1), and $m$ and $m-1$, 
respectively, in cases (2a) and (2b).

The equivalence class of a regular quadratic form $Q$ on $V$ of 
even dimension is also determined by the so-called \emph{Arf 
invariant}, $\Arf(Q)$, which has an especially simple 
definition for the field $\FF_2$: it tells us whether there are
more elements $v\in V$ with $Q(v) = 0$ (isotropic) or 
$Q(v) = 1$ (nonisotropic).
In case (1) half of the elements are isotropic, so $\Arf(Q)$ is 
undefined, while in case (2a) there are $\binom{2^m+1}{2}$ 
isotropic elements and $\binom{2^m }{2}$ nonisotropic, so   
$\Arf(Q)=0$ (or trivial), and in case (2b) there are 
$\binom{2^m}{2}$ isotropic elements and $\binom{2^m+1 }{2}$ 
nonisotropic, so $\Arf(Q)=1$. 

We will later need the following property of the Arf 
invariant. If $Q_1$ and $Q_2$ are regular quadratic forms on spaces of even dimension, 
then   $\Arf(Q_1\obot Q_2)=\Arf(Q_1)+\Arf(Q_2)$.

Since in characteristic $2$ the polarization of a quadratic 
form is alternating, there is a close relationship between 
symplectic and orthogonal groups. For a regular quadratic form 
$Q$, we have $\Ort(V,Q)\le \SP(V,B_Q)$, where the latter 
denotes the stabilizer of $B_Q$ (which is nondegenerate in even 
dimension). %, although the groups do not coincide. 
Besides, over a perfect field, if $V$ has odd dimension and $W$ 
is a complementary subspace of the one-dimensional radical of 
$B_Q$, then $B_Q\vert_W$ is nondegenerate and, by 
\cite[Theorems~14.1 and 14.2]{Grove}, the map 
$\Ort(V,Q)\to \SP(W,B_Q\vert_W)$, 
$\sigma\to\proj_W\circ\sigma\vert_W$, is an isomorphism. 

\smallskip

%%%%%%%%%%%%%%%%%%%%%%%%%%%%%%%%%%%%%%%%%%%%%%%%%%%%%%%%%%%%%

\subsection{Weyl groups of Lie algebras of the 
\texorpdfstring{$E$}{E} types}

We will now recall why the Weyl groups of the simple Lie 
algebras of types $E_6$, $E_7$ and $E_8$ are closely related to 
orthogonal groups over the field $\FF_2$.
Let $\cL$ be one of these Lie algebras over an algebraically 
closed field $\FF$
of characteristic $0$, let $\cH$ be a Cartan subalgebra
of $\cL$, and $\Phi$ the root system associated to $\cH$. That 
is, we have the root space decomposition
\begin{equation}\label{eq:CartanGrading}
\cL=\cH\oplus\bigl(\bigoplus_{\alpha\in\Phi}\cL_\alpha\bigr),
\end{equation}
where $\cL_\alpha=\{x\in\cL\mid [h,x]=\alpha(h)x\;\forall h\in \cH\}$.
 Let $(.\,|\, .)$
be the bilinear form on $\cH^*$, dual to the bilinear form 
obtained from the restriction of the Killing form to $\cH$, 
normalized so that $(\alpha\,|\,\alpha)=2$ for any 
$\alpha\in\Phi$. This bilinear form
restricts to a $\ZZ$-bilinear form on the root lattice 
$R=\ZZ\Phi$. Consider the associated 
$\ZZ$-valued quadratic 
form $\frac{1}{2}(r\,|\, r)$, whose polarization is 
$(.\,|\, .)$. This quadratic form induces
a quadratic form on the vector space $\oR:=R/2R$ over the field 
of two elements $\FF_2$:
\begin{equation}\label{eq:q}
q\colon \oR\longrightarrow \FF_2,\, \bar{r}\mapsto\frac{1}{2}(r\,|\,r) 
\pmod 2,
\end{equation} 
with polarization $b_q$ induced by the bilinear form $(.\,|\, .)$. 
Note that $q(\bar\alpha)=1$ for any root $\alpha\in\Phi$. 
Also, in the cases at hand, if $\beta\ne\pm\alpha$ then $\bar\beta\neq\bar\alpha$.

By straightforward computations with the Cartan matrix, the 
quadratic form $q\colon\oR\rightarrow \FF_2$ is regular, and we 
get the following possibilities,   depending on the
type of the simple Lie algebra $\cL$:

\begin{description}
\item[$E_8$] $q$ is regular with trivial Arf invariant: 
$\Arf(q)=0$. Moreover, in this case the nonisotropic vectors 
are exactly the vectors $\bar\alpha$ for $\alpha\in \Phi$.

\item[$E_7$] $q$ is regular. As the dimension is odd in this 
case, the radical of the polarization has
dimension $1$: $\rad b_q=\FF_2 a$. In this case, the 
nonisotropic vectors are $a$ and the elements
$\bar\alpha$ for $\alpha\in \Phi$.

\item[$E_6$] $q$ is regular with $\Arf(q)=1$. As for $E_8$, the 
nonisotropic vectors are exactly the vectors $\bar\alpha$ for 
$\alpha\in \Phi$.
\end{description}

The following description of the automorphism groups of the 
root systems, $\Aut(\Phi)$, is known (see, e.g., 
\cite[p.~104]{Wilson}), but we will give a proof for completeness, 
and also because we are going to generalize this result in the next section.
Any automorphism of $\Phi$ preserves $2R$, hence there is an 
induced group homomorphism
\begin{equation}\label{eq:rhoAutPhiGLR}
\rho\colon \Aut(\Phi)\longrightarrow \GL(\oR).
\end{equation}

\begin{proposition}\label{pr:CartanDieudonne}  
Let $\cL$ be a simple Lie algebra of type $E_6$, $E_7$, or 
$E_8$ over an algebraically closed field $\FF$ of 
characteristic $0$.
Then the image of the homomorphism $\rho$ is the orthogonal 
group $\Ort(q):=\Ort(\bar R,q)$, and the kernel is $\{\pm I\}$. 
Thus, the Weyl group is isomorphic to
\begin{itemize}
    \item $\Ort(q)$ in type $E_6$; 
    \item $\Ort(q)\times\mathrm{C}_2$ in type $E_7$; 
    \item a central extension of $\Ort(q)$ by $\mathrm{C}_2$ in type $E_8$. 
\end{itemize}
\end{proposition}
\begin{proof}
Any automorphism of the root system $\Phi$ preserves the 
quadratic form $\frac{1}{2}(.\,|\, .)$ on the
root lattice $R=\ZZ\Phi$, and hence the image of $\rho$ is 
contained in $\Ort(q)$. Now, the Weyl group is generated by the 
reflections relative to roots, and $\rho$ maps them to the 
reflections (sometimes called orthogonal transvections in
characteristic $2$) relative to the nonisotropic vectors of 
$\oR$.
Recall that Cartan-Dieudonn\'e Theorem \cite{Dieudonne} says 
that, with one exception, reflections generate the orthogonal 
group of a regular quadratic form.
To be precise, for even-dimensional regular quadratic forms 
over a field of characteristic $2$, the orthogonal group is 
generated by reflections unless the dimension is $4$, the Arf 
invariant is $0$ and the ground field is $\FF_2$ (see, e.g., 
\cite[Theorem 14.16]{Grove}), and we have dimensions $6$ and 
$8$ in our cases. In odd dimension, the orthogonal group is 
isomorphic to a symplectic group, so the orthogonal group can 
be shown to be generated by reflections using the fact that the 
symplectic group is generated by transvections 
(\cite[Theorem 3.4]{Grove}). This completes the proof of the 
surjectivity of $\rho$.

If $\mu\in\Aut(\Phi)$ induces the identity map on $\oR$, then 
it must map each root $\alpha$ to $\pm\alpha$, and it is clear 
that the sign must be the same for all simple roots, because 
the Dynkin diagram is connected. Hence $\ker\rho=\{\pm I\}$. 

Finally, the Weyl group is $\Aut(\Phi)$ for types $E_7$ and 
$E_8$, whereas for type $E_6$ the Weyl group is a complement of 
$\{\pm I\}$ in $\Aut(\Phi)$. Also, for type $E_7$, $\Aut(\Phi)$ 
is the direct product of $\{\pm I\}$ and the subgroup of 
elements
of determinant $1$, as transformations of the euclidean vector space spanned by the roots. 
\end{proof}

%%%%%%%%%%%%%%%%%%%%%%%%%%%%%%%%%%%%%%%%%%%%%%%%%%%%%%%%%%%%%
\smallskip

\subsection{Gradings on Lie algebras}\label{subsec_gradings}

Though the concepts in this section can be considered in a more 
general setting, we will focus on group gradings on Lie 
algebras, in which case it is natural to consider only gradings by
abelian groups. Actually, for simple Lie algebras there is no 
loss of generality in assuming the grading groups are abelian.
All the  algebras will be considered finite-dimensional over an 
algebraically closed field $\FF$ of 
characteristic zero. Most of the material here is standard, and 
can be consulted, for instance, in \cite[Chapter~1]{EKmon}. The notion and 
results on almost fine gradings are extracted from \cite{EK_almostfine}.

\begin{definition}
    If $G$ is an abelian group, and $\cL$ is a Lie algebra, a 
    $G$-\emph{grading} $\Gamma$ on $\cL$ is a decomposition of 
    $\cL$ into a direct sum of subspaces indexed by $G$,
    \begin{equation}\label{eq_gradu}
    \Gamma:\cL=\bigoplus_{g\in G}\cL_g,
    \end{equation}
    such that $\cL_{g}\cL_{h}\subset\cL_{g+h}$ for all 
    $g,h\in G$. The subspace $\cL_g$ is the \emph{homogeneous 
    component} of degree $g$, and its elements are the 
    \emph{homogeneous elements} of degree $g$.
    %It is allowed that $\cL_{g}=0$.
   \end{definition}

 The \emph{support} of $\Gamma$ is the set 
 $\supp \Gamma:=\{g\in G\mid \cL_{g}\ne0\}$. For instance, for any 
 $G$ and any $\cL$, we can consider the \emph{trivial} $G$-grading on $\cL$, 
 given by $\cL_e=\cL$ and $\cL_g=0$ for all $g\ne e$, with support $\{e\}$.
 The \emph{type} of $\Gamma$ is the sequence $(n_1,\ldots,n_m)$, where $n_i$ 
 is the number of homogeneous components of
 $\Gamma$ of dimension $i$, for $i=1,\ldots,m$, with $m$ the maximum
 of these dimensions. Note that $\sum_{1\leq i\leq m}n_ii=\dim\cL$.

Two operations on gradings on $\cL$ are quite natural: 
coarsening and refinement. In general, $\Gamma'$ is said to be a 
\emph{coarsening} of $\Gamma$ if any homogeneous component of 
$\Gamma$ is contained in some homogeneous component of 
$\Gamma'$. The coarsening  is said to be  proper   if at least one of such 
inclusions is proper. 
For each group homomorphism $\alpha\colon G\to H$, we obtain 
 a coarsening ${}^\alpha\Gamma$ of the grading $\Gamma$ in 
\eqref{eq_gradu} as follows: 
\[
{}^\alpha\Gamma:\cL=\bigoplus_{h\in H}{}^\alpha\cL_h \quad 
\textrm{  for  } 
\quad{}^\alpha\cL_h\bydef\bigoplus_{g\in\alpha^{-1}(h)}\cL_g.
\]
Note that ${}^\alpha\Gamma$ is a proper coarsening of $\Gamma$ if 
and only if the restriction of the map $\alpha$ to 
$\supp \Gamma$ is not injective.

If $\Gamma'$ is a coarsening (respectively, a proper 
coarsening) of $\Gamma$, it is also said that $\Gamma$ is a 
\emph{refinement} (respectively, a  proper refinement) of 
$\Gamma'$.    A grading is said to be \emph{fine} if it admits 
no proper refinement.

\begin{example}\label{ex_sl1}
The simple Lie algebra $\cL=\mathfrak{sl}(2,\FF)$ has the following grading by
$G=\ZZ_2^2$:
\[
\Gamma_1: \quad \cL_{(\bar0,\bar1)}=
\FF\left(\begin{smallmatrix}0&1\\1&0\end{smallmatrix}\right),\quad
\cL_{(\bar1,\bar1)}=\FF\left({\begin{smallmatrix}1&0\\0&-1\end{smallmatrix}}
\right),\quad
\cL_{(\bar1,\bar0)}
=\FF\left(\begin{smallmatrix}0&1\\-1&0\end{smallmatrix}\right).
\]
Since all the homogeneous components have dimension 1, $\Gamma_1$ is clearly fine.
\end{example}

\begin{definition}\label{def_equivalent}
   Two gradings on $\cL$,  $\Gamma:\cL=\bigoplus_{g\in G}\cL_g$ and 
   $\Gamma':\cL=\bigoplus_{h\in H}\cL'_h$, are said to be \emph{equivalent} 
   if there is an automorphism $\varphi\in\Aut(\cL)$  and a bijection 
   $\alpha_{\varphi}\colon \supp \Gamma\to \supp \Gamma'$   such 
   that $\varphi(\cL_g)=\cL'_{\alpha_{\varphi}(g)}$ for all 
   $g\in \supp \Gamma$.
   In case $G=H$, $\Gamma$ and $\Gamma'$ are said to be \emph{isomorphic} if 
   there exists $\varphi\in\Aut(\cL)$ with $\alpha_{\varphi}=\id$,
   i.e., $\varphi(\cL_g)=\cL'_g$ for all $g\in G$.  
   \end{definition}

\begin{example}\label{ex_sl23}
The following (fine) gradings on the Lie algebra 
$\cL=\mathfrak{sl}(2,\FF)$ are equivalent: the $\ZZ$-grading
\[
\Gamma_2: \quad
\cL_{-1}=\FF\left(\begin{smallmatrix}0&0\\1&0\end{smallmatrix}\right),\quad
\cL_{0}=\FF\left(\begin{smallmatrix}1&0\\0&-1\end{smallmatrix}\right),\quad
\cL_{1}=\FF\left(\begin{smallmatrix}0&1\\0&0\end{smallmatrix}\right),
\]
and the $\ZZ_3$-grading
\[
\Gamma_3: \quad \cL_{\bar 2}
=\FF\left(\begin{smallmatrix}0&0\\1&0\end{smallmatrix}\right),\quad
\cL_{\bar 0}
=\FF\left(\begin{smallmatrix}1&0\\0&-1\end{smallmatrix}\right),\quad
\cL_{\bar 1}
=\FF\left(\begin{smallmatrix}0&1\\0&0\end{smallmatrix}\right).
\]
\end{example}

As we see from this example, the same set of subspaces of $\cL$ can be regarded as a grading by 
different groups $G$, even if $G$ is generated by the support. 
It turns out that $\ZZ$ is the ``natural'' group for this grading in the sense that it is not only 
generated by the support,
but also the set of relations imposed on the generators is the minimum required to satisfy the definition of grading.
More precisely, consider the grading \eqref{eq_gradu} as a grading by the set $S\bydef\supp\Gamma$, 
$\cL=\bigoplus_{s\in S}\cL_s$, and let
\[
S_*\bydef\{(s_1,s_2)\in S\times S\mid [\cL_{s_1},\cL_{s_2}]\ne 0\}.
\]
For any $(s_1,s_2)\in S_*$, denote by $s_1*s_2$ the unique element of $S$ such that 
$[\cL_{s_1},\cL_{s_2}]\subset\cL_{s_1*s_2}$.

\begin{definition}[\cite{PaZa}]\label{df:univ_group_by_relations}
The abelian group $U=U(\Gamma)$ generated by the set $S$ subject only to the relations 
$s_1+s_2=s_1*s_2$ for $(s_1,s_2)\in S_*$, 
is called the \emph{universal group} of the grading $\cL=\bigoplus_{s\in S}\cL_s$.
\end{definition}

By construction, whenever $S$ is embedded in a group $G$ to make $\cL=\bigoplus_{s\in S}\cL_s$ a $G$-grading, 
there is a unique group homomorphism $U\to G$ that extends the embedding $S\to G$.
This universal property determines $U$ up to isomorphism.

In Example~\ref{ex_sl23}, the universal group is $U(\Gamma_2)=U(\Gamma_3)=\ZZ$ 
(equivalent gradings have isomorphic universal groups). In Example~\ref{ex_sl1}, $U(\Gamma_1)=\ZZ_2^2$.

The universal group also has the following key property. For 
any coarsening $\Gamma':\cL=\bigoplus_{h\in H}\cL'_h$ of $\Gamma$, there exists a unique 
group homomomorphism $\alpha\colon U\to H $ such that 
$\Gamma'={}^\alpha\Gamma$. It is clear that every grading is a 
coarsening of at least one fine grading. As mentioned in the Introduction, this can be used 
to obtain all possible group gradings on a fixed Lie algebra starting from 
the classification of fine gradings up to equivalence, although there is a problem due to the 
lack of uniqueness.

\begin{example}\label{ex_sl4}
Consider again $\cL=\mathfrak{sl}(2,\FF)$, endowed with the $\ZZ_2$-grading
\[
\Gamma_4:  \quad
\cL_{\bar0}
=\FF\left(\begin{smallmatrix}1&0\\0&-1\end{smallmatrix}\right),\quad
\cL_{\bar 1}
=\FF\left(\begin{smallmatrix}0&1\\0&0\end{smallmatrix}\right)\oplus 
\FF\left(\begin{smallmatrix}0&0\\1&0\end{smallmatrix}\right).
\]
Thus $\Gamma_4={}^\alpha\Gamma_1$ for $\alpha\colon \ZZ_2^2\to \ZZ_2$, 
$(\bar a,\bar b)\mapsto \bar a+\bar b$, 
and also $\Gamma_4={}^\beta\Gamma_2$ for 
$\beta\colon \ZZ\to \ZZ_2$, $a\mapsto \bar a$.
\end{example}

In order to assign to each grading only one grading that is ``fine'' in some sense 
(see Definition~\ref{def_almostfine}) we need first to review some groups attached to a grading. 

\begin{definition}[\cite{PaZa}]
    Given a grading $\Gamma$ as in \eqref{eq_gradu}, consider the groups
    \begin{itemize}
        \item $\Diag(\Gamma)\bydef\{\varphi\in\Aut(\cL)\mid \forall g\in G\;\varphi\vert_{\cL_g}\in\FF^\times\id\}$;
        \item $\Stab(\Gamma)\bydef\{\varphi\in\Aut(\cL)\mid \forall g\in G\;\varphi(\cL_g)\subseteq \cL_g\}$;
        \item $\Aut(\Gamma)\bydef\{\varphi\in\Aut(\cL)\mid
                         \forall g\in G\;\exists h\in G\;\varphi( \cL_g)\subseteq \cL_h\}$.
    \end{itemize}
   Note that $\Diag(\Gamma)\trianglelefteq\Stab(\Gamma)\trianglelefteq\Aut(\Gamma)$. 
   The   \emph{Weyl group} of the grading $\Gamma$ is defined by 
$W(\Gamma):=\Aut(\Gamma)/\Stab(\Gamma)$. 
\end{definition}

The Weyl group $W(\Gamma)$ provides a measure of the symmetry of the grading $\Gamma$. 
Note that, for the root space decomposition of a semisimple Lie algebra regarded as a grading 
by the root lattice, this definition gives $W(\Gamma)\cong\Aut(\Phi)$, which may be larger 
than the classical Weyl group.

The grading $\Gamma$ (more precisely, the set of its homogeneous components) 
can be recovered as the simultaneous eigenspace decomposition with respect 
to  $\Diag(\Gamma)$. This  is a diagonalizable algebraic group, hence 
isomorphic to the direct product of a torus and a finite abelian group, 
usually called a \emph{quasitorus}. Recall that the rank of an algebraic 
group is the dimension of any of its maximal tori (which are known to be 
conjugate over an algebraically closed field). Remarkably, this diagonal 
group is isomorphic to the group of characters 
$\widehat{U}=\Hom(U,\FF^\times)$ 
of the universal group $U=U(\Gamma)$. In fact, the map which sends any $\chi\in\widehat{U}$ 
to the automorphism $\varphi_\chi\in\Aut(\cL)$  defined by 
$\varphi_\chi\vert_{\cL_g}=\chi(g)\,\id$ for all $g\in U$, is an isomorphism.
If $U\cong\ZZ^l\times H$ for $H$ a finite abelian group, then 
$\Diag(\Gamma)\cong (\FF^\times)^l\times H$ has dimension $l$, equal to the rank of the universal group. 

Each $\varphi\in\Aut(\Gamma)$  produces a grading isomorphic to $\Gamma$, given by 
$\cL'_g\bydef\varphi(\cL_g)$, and the bijection $\alpha_\varphi$ in Definition~\ref{def_equivalent} 
is naturally extended to a group automorphism of the universal group $U=U(\Gamma)$. 
This gives a group homomorphism $\Aut(\Gamma)\to \Aut(U)$,
whose kernel is precisely $\Stab(\Gamma)$, so the Weyl group of the grading can be seen as a subgroup of 
$\Aut(U)$. 

\begin{example}\label{ex_grupos}
    Consider, for $\cL=\mathfrak{sl}(2,\FF)$ and each $c\in\FF^\times$, the automorphism 
    $\varphi_c\colon\cL\to\cL$ defined by 
    $\varphi_c\left(\begin{smallmatrix}x&y\\z&-x\end{smallmatrix}\right)
     =\left(\begin{smallmatrix}x&cy\\c^{-1}z&-x\end{smallmatrix}\right)$ and the one-dimensional torus 
     $\cT=\{ \varphi_{c}\mid c\in\FF^\times\}$. Take also the order two automorphism $\sigma$ of $\cL$ given by
    $\sigma\left(\begin{smallmatrix}x&y\\z&-x\end{smallmatrix}\right)
      =\left(\begin{smallmatrix}-x&-z\\-y&x\end{smallmatrix}\right)$. 
    
    For $\Gamma_4$ as in Example~\ref{ex_sl4}, it is easily computed that 
    $\Diag(\Gamma_4)=\{\id,\varphi_{-1}\}$ while 
    $\Stab(\Gamma_4)=\Aut(\Gamma_4)=\cT\cup\cT\sigma\cong\cT\rtimes\mathrm{C}_2$, since 
    $\sigma\varphi_c=\varphi_{c^{-1}}\sigma$.
     In particular, the Weyl group $W(\Gamma_4)$ is trivial.  
    
    For $\Gamma_2$ as in Example~\ref{ex_sl23}, the automorphism $\sigma$
    above permutes the homogeneous components. Note that $\Diag(\Gamma_2)=\Stab(\Gamma_2)=\cT$, 
    while $\Aut(\Gamma_2)\cong\cT\rtimes\mathrm{C}_2$. Here $W(\Gamma_2)\cong\mathrm{C}_2$.

    For $\Gamma_1$ as in Example~\ref{ex_sl1}, 
    $\Diag(\Gamma_1)=\Stab(\Gamma_1)=\langle \varphi_{-1},\sigma\rangle$ is isomorphic to the 
    universal group $\ZZ_2^2$, while the Weyl group 
    $W(\Gamma_1)=\Aut(\ZZ_2^2)\cong \mathrm{S}_3$, since we can easily find 
    automorphisms interchanging any pair of homogeneous components.
\end{example}

The \emph{toral rank} of $\Gamma$ is defined as the rank of the algebraic group $\Stab(\Gamma)$, 
which is denoted by $\mathrm{tor.rank}(\Gamma)$. Thus the toral rank is at least the rank of the 
finitely generated abelian group $U(\Gamma)$. For instance, the toral rank of $\Gamma_4$ is 1, 
although the rank of the universal group is 0.

\begin{definition}\label{def_almostfine}
    A grading $\Gamma$ is \emph{almost fine} if $\rank(U(\Gamma))=\mathrm{tor.rank}(\Gamma)$ 
    or, in other words, if the maximal torus contained in $\Diag(\Gamma)$ is also maximal in $\Stab(\Gamma)$.
\end{definition}

If $\Gamma$ is a fine grading, then $\Diag(\Gamma)$ is a maximal quasitorus of $\Aut(\cL)$ and 
hence $\Gamma$  is almost fine, but the converse is not true. There are no examples of this fact 
in the case of $\cL=\mathfrak{sl}(2,\FF)$, where $\Gamma_1$ and $\Gamma_2$ are the only 
almost fine gradings up to equivalence, 
but we will later see for $\cL$ of the $E$ types several examples of almost fine gradings that are not fine.

If $\Gamma'$ is a coarsening of $\Gamma$, then the automorphisms of $\cL$ preserving the 
homogeneous components of $\Gamma$ will also preserve those of $\Gamma'$, so    
$\Stab(\Gamma)\subseteq \Stab(\Gamma')$. 
This implies that refinements do not increase the toral rank. 
For almost fine gradings, any refinement preserves the toral rank. 

For semisimple Lie algebras, an almost fine grading admits  quite a simple characterization: 
$\rank(U(\Gamma))=\dim\cL_e$.  We always
have $\rank(U(\Gamma))\leq \dim\cL_e$ so, in particular, $\cL_e=0$ if and
only if $\Gamma$ is almost fine and has finite universal group.
 The gradings satisfying $ \cL_e=0$ are called \emph{special} gradings. 
 One example that appeared earlier is $\Gamma_1$.
In a special grading, every homogeneous element is semisimple
\cite[3.6 Proposition]{Hesselink}, which gives another reason to study 
them: they provide 
nice bases of semisimple Lie algebras.

For our purposes, the reason to consider almost fine gradings in detail is that every grading admits 
a canonical almost fine refinement (thus avoiding duplications in the classification process). 
More precisely, if $\Gamma$ is the $G$-grading in \eqref{eq_gradu}, and we pick a maximal 
torus $\cT$ in $\Stab(\Gamma)$, consider the refinement $\Gamma^*$ obtained by 
decomposing each homogeneous component $\cL_g$ with respect to the action of $\cT$:
\begin{equation}\label{eq_gradu*}
    \Gamma^*:\cL
    =\hspace*{-12pt}\bigoplus_{(g,\lambda)\in G\times\widehat{\cT}}\hspace*{-12pt}\cL_{(g,\lambda)},
    \textrm{ for }\cL_{(g,\lambda)}:=\{x\in\cL_g\mid \tau(x)=\lambda(\tau)x\ \forall \tau\in\cT\}.
    \end{equation} 
    The grading $\Gamma^*$ is almost fine, with the same toral rank as $\Gamma$, 
    and any other grading satisfying these properties is equivalent to $\Gamma^*$. %(CREO).
    In the above examples, $\Gamma_4$ is not almost fine and $\Gamma_4^*=\Gamma_2$. 
    After this discussion, it becomes clear that the idea for achieving a classification consists 
    in studying the coarsenings of almost fine gradings that preserve their toral rank. The suitable 
    coarsenings are given by \emph{admissible} homomorphisms 
    $\alpha\colon U\to H$ as in \cite[\S4]{EK_almostfine}.

A grading is \emph{pure} if there is a Cartan subalgebra contained in some homogeneous 
component distinct from the neutral homogeneous component.
For instance, $\Gamma_1$ and $\Gamma_4$ are pure gradings, but only $\Gamma_1$ is special. 
In this paper, we will focus on special pure gradings, since 
they are almost fine and serve as the first step towards a classification of gradings on 
the simple Lie algebras of the $E$ types.

There is a nice  special pure  grading on any semisimple   $\cL$, given by the following example.

\begin{example}\label{ex_finas}
  Let $\cL$ be a semisimple Lie algebra of rank $l$. Consider  
  a Cartan subalgebra $\cH$ of $\cL$, the root system $\Phi$ associated to $\cH$, 
  and the root space decomposition as in \eqref{eq:CartanGrading}. As in 
  Example~\ref{ex_grupos}, there exists an order
$2$ automorphism $\sigma$ of $\cL$ which is $-\id$ on $\cH$, and hence
takes any root space $\cL_\alpha$ to $\cL_{-\alpha}$. Then $\sigma$
commutes with the elements of order $2$ in the torus $\cT$ associated 
to $\cH$. The eigenspaces of the quasitorus $\cT_2\times\langle\sigma\rangle$, where 
$\cT_2\bydef\{\tau\in\cT\mid \tau^2=\id\}$, are the homogeneous components of a
 special pure grading on $\cL$, whose universal group is isomorphic to $\ZZ_2^{l+1}$. 
\end{example}

We are going to generalize this example in the next section, where it will correspond to the case 
$E=2R$. The homogeneous components are given by \eqref{eq:homogeneousGamma}.
For a simple Lie algebra $\cL$ with a simply laced root system, the nonzero 
homogeneous components of this grading have dimension $1$, with the sole exception of the
Cartan subalgebra $\cH$, which implies that the grading is fine. For $\frsl(2,\FF)$, it 
is the grading in Example~\ref{ex_sl1}.

It turns out that all the special pure gradings are coarsenings of the
grading in this example, as we will see in Section~\ref{se:SpecialPure}.

\bigskip

%%%%%%%%%%%%%%%%%%%%%%%%%%%%%%%%%%%%%%%%%%%%%%%%%%%%%%%

\section{Weyl groups of special pure gradings}\label{se:SpecialPure}

This section will be devoted to describing the special pure gradings of semisimple Lie algebras and to 
begin our study of the Weyl groups, following
the work of Hesselink \cite{Hesselink}.

\smallskip

Let $\cL$ be a finite-dimensional semisimple Lie algebra over an algebraically closed ground field $\FF$ of
characteristic $0$,
 and let 
\[
 \Gamma: \cL=\bigoplus_{g\in G}\cL_g
\] 
 be a pure grading, i.e., there is an element 
$s\in G\setminus\{e\}$ such that $\cL_s$ contains a Cartan subalgebra $\cH$ of $\cL$. Fix the 
element $s$ in what follows.

Assume, from now on, that $G$ is the universal group of $\Gamma$. 
Recall that its 
group of (multiplicative) characters 
$\widehat G$ is
 isomorphic to the diagonal group $\Diag(\Gamma)$, where $\chi\in\Hom(G,\FF^\times)$ corresponds to
the diagonal automorphism $\tau_\chi\colon x\in\cL_g\mapsto \chi(g)x$ for any $g\in G$.

Consider the torus associated to the Cartan subalgebra $\cH$:
\[
\cT=\{\tau\in\Aut(\cL)\mid \tau\vert_\cH=\id\}.
\]
Then (see, e.g., \cite[Proposition 2.2]{EldNonGroup}), 
the diagonal group of $\Gamma$ is
\begin{equation}\label{eq:DiagGamma}
\Diag(\Gamma)=\bigl(\Diag(\Gamma)\cap \cT_2\bigr)\times\langle \sigma\rangle
\end{equation}
for some element $\sigma\in\Aut(\cL)$ such that $\sigma^2=\id$ and $\sigma\vert_\cH=-\id$, and 
where $\cT_2$ is 
the $2$-periodic part of $\cT$: $\cT_2=\{\tau\in\cT\mid \tau^2=\id\}$. In particular, 
$\Diag(\Gamma)$ is an elementary abelian $2$-group, and so is $G$.

For the sake of completeness, we include a sketch of the proof of \eqref{eq:DiagGamma}. 
If $\chi$ is a character of $G$ with 
$\chi(s)\ne 1$, the automorphism $\sigma=\tau_\chi$ restricts
to $\chi(s)\id$
on $\cL_s$ and hence on $\cH$. This forces $\chi(s)^{-1}\alpha$ to be a root for any root $\alpha$
relative to $\cH$. It follows that $\chi(s)=-1$, which forces $\sigma^2=\id$ 
(\cite[Lemma 2.1]{EldNonGroup}). This same argument shows 
that $\chi(s)=1$ or $-1$ for 
any $\chi\in\widehat G$, and hence $\Diag(\Gamma)$ is contained in the centralizer
of $\sigma$ in $\{\varphi\in\Aut(\cL)\mid \varphi\vert_\cH=\pm \id\}=\cT\rtimes\langle\sigma\rangle$,
and this centralizer is $\cT_2\times\langle\sigma\rangle$ (\cite[Lemma 2.1]{EldNonGroup}). It follows
that $\Diag(\Gamma)=\bigl(\Diag(\Gamma)\cap \cT_2\bigr)\times\langle \sigma\rangle$.

Let $\Phi$ be the root system associated to $\cH$: 
\begin{equation*}%\label{eq:CartanGrading}
\cL=\cH\oplus\bigl(\bigoplus_{\alpha\in\Phi}\cL_\alpha\bigr),
\end{equation*}
and let $R=\ZZ \Phi$ be the root lattice. Write, as before, 
$\overline{R}\bydef  R/2R$. Then, identifying
any character of $\overline{R}$ with a character of $R$ trivial on $2R$, we have 
\[
\cT_2=\{\tau_\chi\mid \chi\in \Hom(\oR,\{\pm 1\})\},
\] 
where 
$\tau_\chi\colon x\in\cL_\alpha\mapsto \chi(\overline{\alpha})x$ for any $\alpha\in\Phi\cup\{0\}$.

\begin{remark}\label{re:involution_types}
Recall that, for $\cL$ of types $E_6$, $E_7$, $E_8$, we have defined in \eqref{eq:q} a 
regular quadratic form $q$ on $\oR$, regarded as a vector space over $\FF_2$.
Since $\cT_2$ is identified with the character group of $\oR$, which can be considered as 
the dual space of $\oR$, we can transport the quadratic form $q$ in the cases $E_6$ and 
$E_8$ along the isomorphism $\oR\to\cT_2$ defined by the nondegenerate bilinear form $b_q$. 
The resulting quadratic form $\hat{q}$ on $\cT_2$ has the following meaning. For $\cL$ of type $E_8$, 
there are two conjugacy classes of automorphisms of order $2$, which are distinguished by the 
subalgebra of fixed points, and $\hat{q}$ takes value $0$ on the nontrivial elements $\tau\in\cT_2$ 
for which $\cL^\tau$ is of type $D_8$ and value $1$ on those for which $\cL^\tau$ is of type $E_7+A_1$. 
Moreover, the elements of $\cT_{2}\times\langle\sigma\rangle$ that are not in $\cT_2$ give type $D_8$. 
For $\cL$ of type $E_6$, there are two conjugacy classes of inner automorphisms of order $2$, and 
$\hat{q}$ takes value $0$ when $\cL^\tau$ is of type $D_5+Z$ (where $Z$ stands for the 
$1$-dimensional center) and value $1$ when $\cL^\tau$ is of type $A_5+A_1$. 
Moreover, the elements of $\cT_{2}\times\langle\sigma\rangle$ that are not in $\cT_2$ are 
outer automorphisms (with $\cL^\tau$ of type $C_4$).
\end{remark}

Now consider the quotient of $G$ by the subgroup generated by the fixed element $s$: 
$\oG=G/\langle s\rangle$, and the associated coarsening:
\[
\oGamma:\cL=\bigoplus_{\bar g\in\oG}\cL_{\bar g},
\]
where, for any $g\in G$, $\bar g$ denotes its coset $g\langle s\rangle$, and 
$\cL_{\bar g}=\cL_g\oplus\cL_{gs}$. As the Cartan subalgebra $\cH$ is contained in the neutral
homogeneous component $\cL_{\bar s}=\cL_{\bar e}$, the grading $\oGamma$ is a coarsening of the
Cartan grading \eqref{eq:CartanGrading} and this gives a surjective group homomorphism
\begin{equation}\label{eq:pi}
\pi\colon R\longrightarrow \oG
\end{equation}
such that for any root $\alpha\in\Phi$, $\pi(\alpha)=\bar g$ if $\cL_\alpha\subseteq \cL_{\bar g}$.

Let $E\bydef \ker\pi$ be the kernel of $\pi$, which is called the \emph{complementary lattice} in 
\cite[1.4]{Hesselink}. 
As $G$ is $2$-elementary, $2R$ is contained in $E$. Also, $\pi$ induces an isomorphism
\begin{equation}\label{eq:piE}
\pi_E\colon R/E\longrightarrow \oG.
\end{equation}

Besides, $G$ being the universal group of $\Gamma$, and $\oGamma$ being a proper coarsening of 
$\Gamma$, there is a unique surjective map of $G$ onto the universal group of $\oGamma$, with 
$s$ in its kernel, and hence, $\oG$ is the universal group of $\oGamma$. We therefore have:
\begin{equation}\label{eq:Ecirc}
E= 2R +\ZZ(\Phi^+\cap E)+\ZZ\{\alpha-\beta\mid \alpha,\beta\in\Phi^+\ 
\text{and}\ \alpha-\beta\in E\},
\end{equation}
or, in the notation of 
\cite[(2.4)]{EldNonGroup}, $E=E^\circ$. This is because the subgroup on
the right hand side of \eqref{eq:Ecirc} is contained in $E$, and the quotient
of $R$ by this subgroup gives the same coarsening $\overline\Gamma$, up to equivalence. (For
details, see \cite[\S2]{EldNonGroup}.)

Moreover, the diagonal group of $\oGamma$ is
\[
\Diag(\oGamma)=\Diag(\Gamma)\cap \cT_2=\{\tau_\chi\mid \chi\in\Hom(R/E,\{\pm 1\})\}.
\]

A word of caution is needed here. Hesselink does not consider universal groups, although the grading
groups appearing in \cite{Hesselink} are universal in many cases.

The automorphism $\sigma$ in \eqref{eq:DiagGamma} is not unique. If we fix one of them, and consider
the corresponding character $\chi\in \widehat G$, so that $\chi(g)=\pm 1$ according to 
$\sigma\vert_{\cL_g}=\pm\id$, we can identify our grading group $G$ with $R/E\times \FF_2$:
\begin{equation}\label{eq:GRE2}
\begin{split}G&\simeq R/E\times \FF_2\\
                 g&\leftrightarrow \bigl(\pi_E^{-1}(\bar g), i\bigr),
\end{split}
\end{equation}
where $i= 0$ if $\chi(g)=1$, and $i= 1$ if $\chi(g)=-1$. Under this identification, the element $s$ 
corresponds to $(0, 1)\in R/E\times\FF_2$
and the homogeneous components of $\Gamma$ are the following (\cite[(2.3)]{EldNonGroup}):
\begin{equation}\label{eq:homogeneousGamma}
\begin{split}
\cL_{(r+E, 0)}
 &=\bigoplus_{\alpha\in\Phi^+\cap(r+E)}\FF\left(x_\alpha+\sigma(x_\alpha)\right),\\
\cL_{(r+E, 1)}
  &=\begin{cases}
   \bigoplus_{\alpha\in\Phi^+\cap(r+E)}\FF\left(x_\alpha-\sigma(x_\alpha)\right)
       &\text{if $r+E\neq E$,}\\
  \cH\oplus\left(\bigoplus_{\alpha\in\Phi^+\cap(r+E)}
           \FF\left(x_\alpha-\sigma(x_\alpha)\right)\right)
           &\text{if $r+E=E$,}
        \end{cases}
\end{split}
\end{equation}
while the homogeneous components of $\oGamma$ are the following:
\begin{equation}\label{eq:homogeneousGammabar}
\begin{split}
\overline\cL_{r+E}
 &=\bigoplus_{\alpha\in\Phi^+\cap(r+E)}\left(\cL_\alpha\oplus\cL_{-\alpha}\right),\qquad\text{if $r\not\in E$,}\\
\overline\cL_{E}&=
  \cH\oplus\left(\bigoplus_{\alpha\in\Phi^+\cap E}\left(\cL_\alpha\oplus\cL_{-\alpha}\right)\right).
\end{split}
\end{equation}
Here we have fixed a system of simple roots, which decomposes $\Phi$ into the set of positive roots 
$\Phi^+$ and negative roots $\Phi^-=-\Phi^+$, and have picked nonzero elements $x_\alpha\in\cL_\alpha$
for any $\alpha\in\Phi^+$.

Equation~\eqref{eq:homogeneousGamma} gives at once the following result.

\begin{proposition}\label{pr:Phi_cap_E}
Let $\Gamma$ be a pure  grading on a semisimple Lie algebra $\cL$ 
over an algebraically closed field of characteristic $0$, with 
universal group
$G$. Let $s\in G$ be an element such that $\cL_s$ contains a Cartan subalgebra $\cH$. Let $\Phi$ be the
root system relative to $\cH$, $R=\ZZ\Phi$ the root lattice, and  $E$ the complementary lattice of $\Gamma$. Then $\Gamma$ 
is special if and only if $\Phi\cap E=\emptyset$.
\end{proposition}

\smallskip

Our aim in this section is, assuming that $\Gamma$ is also special, to compute the stabilizer 
$W_s(\Gamma)$ of $s$ in the Weyl group of $\Gamma$:
\[
W_s(\Gamma)=\{ w\in W(\Gamma)\mid w(s)=s\}.
\]
In many situations, there is a unique $s\in G$ with $L_s$ containing a Cartan subalgebra, so $W_s(\Gamma)$ is the
whole Weyl group $W(\Gamma)$.

Because of \eqref{eq:GRE2}, we may identify $W(\Gamma)$ with a subgroup of 
$\GL(R/E\times \FF_2)\simeq \GL_{l+1}(2)$, where $l$ is the dimension of $R/E$, which is a 
$2$-elementary group,
and hence a vector space over $\FF_2$. As $s$ corresponds to $(0,1)$, $W_s(\Gamma)$ is contained in the
following subgroup:
\begin{equation}\label{eq:GLRE_HomREF2}
\left(
\begin{array}{c|c}
\GL(R/E)&\begin{matrix} 0\\[-3pt] \vdots \\ 0\end{matrix}\\
\hline\\[-8pt]
\Hom(R/E,\FF_2) &1
\end{array}
\right),
\end{equation}
a semidirect product of $\Hom(R/E,\FF_2)$ and $\GL(R/E)$.

\begin{lemma}\label{le:WsGamma}
Let $\Gamma$ be a special pure grading on the semisimple Lie algebra $\cL$.
Under the identification above, the subgroup
\[
\left(
\begin{array}{c|c}
I_l&\begin{matrix} 0\\[-3pt] \vdots \\ 0\end{matrix}\\
\hline\\[-8pt]
\Hom(R/E,\FF_2) &1
\end{array}
\right),
\]
is contained in $W_s(\Gamma)$.
\end{lemma}

\begin{proof}
Let $\xi$ be an arbitrary element of $\Hom(R/E,\FF_2)$, and let 
$\tilde\xi\colon R\rightarrow \{\pm 1\}$ be the
character of $R$ given by $\tilde\xi(q)=(-1)^{\xi(r+E)}$. As $\FF$ is algebraically closed, let 
$\chi\in\Hom(R,\FF^\times)$ be a character with $\chi^2=\tilde\xi$, and consider the corresponding 
automorphism $\tau_\chi\in\cT$. For any $\alpha\in \Phi^+$, and $x_\alpha\in\cL_\alpha$,
\[
\tau_\chi(x_\alpha+\sigma(x_\alpha))=\chi(\alpha)x_\alpha +\chi(\alpha)^{-1}\sigma(x_\alpha)
\in\begin{cases}
 \cL_{(\alpha + E,0)}&\text{if $\chi^2(\alpha)=\tilde\xi(\alpha)=1$,}\\
 \cL_{(\alpha + E,1)}&\text{if $\chi^2(\alpha)=\tilde\xi(\alpha)=-1$.}
 \end{cases}
 \]
 Moreover, if $\beta\in\Phi^+$ is another positive root with $\beta-\alpha\in E$, then 
 $\chi^2(\alpha)=\chi^2(\beta)$, because $\tilde\xi(E)=1$. Therefore, because of \eqref{eq:homogeneousGamma},
  $\tau_\chi$ is an automorphism of $\Gamma$ such that 
  $\tau_\chi(\cL_{(r+E,0)})$ is contained in 
 $\cL_{(r+E,\xi(r+E))}$ for any $r\in R$, so that its projection on $W(\Gamma)$ is
 \[
 \left(
 \begin{array}{c|c}
\quad I_l\quad &\begin{matrix} 0\\[-3pt] \vdots \\ 0\end{matrix}\\
\hline\\[-8pt]
\xi &1
\end{array}
\right),
\]
thus proving our result.
\end{proof}

Any automorphism of the root system $\Phi$ preserves $2R$, and hence the induced group 
homomorphism in \eqref{eq:rhoAutPhiGLR}:
\begin{equation*} 
\rho\colon \Aut(\Phi)\longrightarrow \GL(\oR),
\end{equation*}
gives an action of $\Aut(\Phi)$ on $\oR$. Write $\oE=E/2R$.

Any $\mu\in \Stab_{\Aut(\Phi)}(\oE)$ induces an automorphism of $R/E\simeq \oR/\oE$. 
Thus $\rho$ induces a group homomorphism
\begin{equation}\label{eq:rho'}
\begin{split}
\rho'\colon \Stab_{\Aut(\Phi)}(\oE)&\longrightarrow \GL(R/E)\\
   \mu&\mapsto \mu'.
\end{split}
\end{equation}
Consider the composition 
\begin{equation}\label{eq:upsilon}
\upsilon\colon G\xrightarrow{\phantom{nat}}\oG\xrightarrow{\pi_E^{-1}}R/E,
\end{equation} 
where the first map is the natural homomorphism and the second is the inverse of the isomorphism
$\pi_E$ in \eqref{eq:piE}.
It induces a surjective group homomorphism 
\[
\tilde\upsilon\colon \Stab_{\Aut(G)}(s)\rightarrow \GL(R/E).
\]
We get the diagram 
\[
\begin{tikzcd}
W_s(\Gamma) \arrow[r, hook] &\Stab_{\Aut(G)}(s) \arrow[d, "\tilde\upsilon"]\\
\Stab_{\Aut(\Phi)}(\oE) \arrow[r, "{\rho'}"] & \GL(R/E)  
\end{tikzcd}
\]

The next result extends \cite[4.8]{Hesselink}.

\begin{theorem}\label{th:WsGamma}
Let $\Gamma$ be a special pure grading on the semisimple Lie algebra $\cL$,
over an algebraically closed field of characteristic $0$, with universal  group
$G$. Let $s\in G$ be an element such that $\cL_s$ contains a Cartan subalgebra $\cH$. Then, with
the notations above, the stabilizer of $s$ in the Weyl group $W(\Gamma)$ is obtained as follows:
\[
W_s(\Gamma)={\tilde\upsilon}^{-1}(\im \rho'). 
\]
In other words, $W_s(\Gamma)$ consists of the matrices in 
\eqref{eq:GLRE_HomREF2} with arbitrary lower left corner, and where
the upper left corner (in $\GL(R/E)$) is induced by an automorphism of
$\Phi$ that preserves $\oE$.
\end{theorem}

\begin{proof}
Note that $\cL_s=\cH$ holds, because of Proposition~\ref{pr:Phi_cap_E} and 
Equation~\eqref{eq:homogeneousGamma}.
We proceed in several steps.

\smallskip

\noindent\textbf{1)}\quad
Any $\varphi\in\Aut(\Gamma)$ preserving $\cL_s=\cH$ induces an automorphism 
$\mu_\varphi\in\Aut(\Phi)$ with $\varphi(\cL_\alpha)=\cL_{\mu_\varphi(\alpha)}$ for any $\alpha\in\Phi$.

Since $\Gamma$ is special, we have $\Phi^+\cap E=\emptyset$. For any $\alpha,\beta\in\Phi^+$ with 
$\alpha-\beta\in E$, the elements $x_\alpha+\sigma(x_\alpha)$ and $x_\beta+\sigma(x_\beta)$ are in
the same homogeneous component of $\Gamma$, and hence so are their images under $\varphi$. Thus
$\mu_\varphi(\alpha)-\mu_\varphi(\beta)$ is in $E$. Equation~\eqref{eq:Ecirc} shows 
$\mu_\varphi(E) \subseteq E$, which is equivalent to $\mu_\varphi\in \Stab_{\Aut(\Phi)}(\oE)$.

Note that the image under $\tilde\upsilon$ of the element of $W_s(\Gamma)$ determined by $\varphi$
is the image under $\rho'$ of $\mu_\varphi$.

\medskip

\noindent\textbf{2)}\quad
Conversely, for any $\mu\in \Stab_{\Aut(\Phi)}(\oE)$,  take an element $\varphi\in \Stab_{\Aut(\cL)}(\cH)$ with $\mu_\varphi=\mu$ (\cite[14.2]{Humphreys}).
Note that we have $\mu(E)=E$ and $\varphi\in\Aut(\oGamma)$ (recall the homogeneous components
of $\oGamma$ in \eqref{eq:homogeneousGammabar}).

For any $\tau\in\Diag(\Gamma)\cap \cT=\Diag(\Gamma)\cap\cT_2=\Diag(\oGamma)$, we get
$\varphi\tau\varphi^{-1}\in\Diag(\oGamma)$. Consider the automorphism 
$\sigma'=\varphi\sigma\varphi^{-1}$, which satisfies $\sigma'\vert_{\cH}=-\id$, so there exists
an element $\tau\in\cT$ such that $\sigma'=\tau\sigma$. Take $\tilde\tau\in\cT$ with $\tilde\tau^2=\tau$.
Then $\sigma'=\tau\sigma=\tilde\tau^2\sigma=\tilde\tau\sigma\tilde\tau^{-1}$. Therefore we have
$\varphi\sigma\varphi^{-1}=\tilde\tau\sigma\tilde\tau^{-1}$ and $\tilde\tau^{-1}\varphi$ commutes
with $\sigma$, stabilizes $\cH$ and satisfies $\mu_{\tilde\tau^{-1}\varphi}=\mu_\varphi=\mu$.

Replacing $\varphi$ by $\tilde\tau^{-1}\varphi$, we may assume $\varphi\sigma=\sigma\varphi$, and 
hence we get $\varphi\in\Aut(\Gamma)$ with $\varphi(\cL_s)=\cL_s$, because the homogeneous
components of $\Gamma$ (Equation~\eqref{eq:homogeneousGamma}) are the eigenspaces for $\sigma$
of the homogeneous components of $\oGamma$.

\medskip

\noindent\textbf{3)}\quad
For any $\varphi\in\Aut(\Gamma)$ preserving $\cL_s=\cH$ we obtain the following elements:
\[
\mu_\varphi\in\Stab_{\Aut(\Phi)}(\oE),\quad
\rho'(\mu_\varphi)=\mu_\varphi'\in \GL(R/E),
\]
and for any $\alpha\in\Phi^+$ we have
\[
\varphi(\cL_{(\alpha+E,0)})=\cL_{(\mu_\varphi'(\alpha+E),0\ \text{or}\ 1)}.
\]
The associated element in $W_s(\Gamma)$ is then, considered as an element of 
$\GL(R/E\times \FF_2)$,  of the form
\[
 \left(
 \begin{array}{c|c}
\quad \mu_\varphi' \quad &\begin{matrix} 0\\[-3pt] \vdots \\ 0\end{matrix}\\
\hline\\[-8pt]
\xi &1
\end{array}
\right),
\]
for some $\xi\in\Hom(R/E,\FF_2)$. But step \textbf{2)} shows that given any element 
$\mu'=\rho'(\mu)$ with $\mu\in \Stab_{\Aut(\Phi)}(\oE)$, there is an element 
$w\in W_s(\Gamma)$ of the form
\[
 \left(
 \begin{array}{c|c}
\quad \mu' \quad &\begin{matrix} 0\\[-3pt] \vdots \\ 0\end{matrix}\\
\hline\\[-8pt]
\xi &1
\end{array}
\right),
\]
for some $\xi\in\Hom(R/E,\FF_2)$. Now Lemma~\ref{le:WsGamma} completes the proof.
\end{proof}

\bigskip

%%%%%%%%%%%%%%%%%%%%%%%%%%%%

\section{Special pure gradings on simple Lie algebras of types 
\texorpdfstring{$E_6$}{E6}, \texorpdfstring{$E_7$}{E7}, or \texorpdfstring{$E_8$}{E8}}\label{se:SpecialpureE}

The aim of this section is the classification, up to equivalence, of the 
special pure
gradings on the simple Lie algebras of type $E_6$, $E_7$, and $E_8$.
We will make use of Proposition~\ref{pr:CartanDieudonne} that relates
the Weyl groups of   these algebras with suitable orthogonal groups
over the field of $2$ elements. 

We start with $E_8$.

\begin{theorem}\label{th:E8}
Up to equivalence, there are exactly five special pure gradings on a simple Lie algebra of type $E_8$
over an algebraically closed field of characteristic $0$. Their 
universal groups and types are given in the next table:
\begin{center}
\begin{tabular}{ccc}
\bigstrut  & Universal group & Type\\
\hline\hline
\bigstrut  $\Gamma_{E_8}^9$ &$\ZZ_2^9$&$(240,0,0,0,0,0,0,1)$\\
\hline
\bigstrut  $\Gamma_{E_8}^8$ &$\ZZ_2^8$&$(128,56,0,0,0,0,0,1)$\\
\hline
\bigstrut $\Gamma_{E_8}^7$ &$\ZZ_2^7$&$(0,96,0,12,0,0,0,1)$\\
\hline
\bigstrut $\Gamma_{E_8}^6$ &$\ZZ_2^6$&$(0,0,0,56,0,0,0,3)$\\
\hline
\bigstrut $\Gamma_{E_8}^5$ &$\ZZ_2^5$&$(0,0,0,0,0,0,0,31)$\\
\hline
\end{tabular}
\end{center}
\end{theorem}
\begin{proof}
The quadratic form $q\colon\oR\rightarrow \FF_2$ in \eqref{eq:q} is regular with trivial Arf invariant. Moreover,
the nonisotropic vectors of $\oR$ are precisely the $120$ vectors $\bar\alpha\bydef\alpha+2R$ for 
$\alpha\in\Phi^+$. By Proposition~\ref{pr:Phi_cap_E},
 the complementary lattice $E$ satisfies that $\oE=E/2R$ is a totally isotropic subspace of 
$\oR$, and hence it has dimension at most $4$. Moreover, for any dimension up to $4$, any two totally 
isotropic subspaces of $\oR$ are conjugate under the action of $\Ort(q)$. As a consequence, we get 
the following possibilities, according to \eqref{eq:homogeneousGamma}:
\begin{description}
\item[$\dim\oE=0$] Here the universal group $R/E\times\FF_2\simeq \oR\times\FF_2$ 
(see \eqref{eq:GRE2}) is isomorphic to $\ZZ_2^9$, the homogeneous space $\cL_{(0,1)}$ is a Cartan 
subalgebra, and all the other homogeneous components are one-dimensional. Hence, the type
is $(240,0,0,0,0,0,0,1)$.

\smallskip

\item[$\dim\oE=1$] In this case the universal  group is isomorphic to $\ZZ_2^8$. There is an
isotropic vector $u$ with $\oE=\FF_2 u$. 
If we fix an 
element $v$ with $b_q(u,v)=1$, then $u$ and $v$ span a hyperbolic plane $H$,
and $\oR=H\obot H^\perp$, so $0=\Arf(q)=\Arf(q\vert_H)+\Arf(q\vert_{H^\perp})=\Arf(q\vert_{H^\perp})$. 
Moreover, $q$ induces a quadratic form on the six-dimensional space
$\oE^\perp/\oE$, which is isometric to $H^\perp$ and hence has
 trivial Arf invariant. Thus, $\oE^\perp/\oE$ contains $28$ nonisotropic vectors. 
 Besides, for any $v\in \oR$ with $q(v)=1$, we have $q(v+u)=1$ if
and only if $b_q(u,v)=0$, that is, $v\in\oE^\perp$, so
 there are $28$ pairs of positive roots $(\alpha,\beta)$ such that $\alpha+E=\beta +E$ 
 ($\alpha+2R=v\in\oE^\perp$, $\beta+2R=u+v$).
This gives $56$ two-dimensional homogeneous components in \eqref{eq:homogeneousGamma}.
For the remaining positive roots, that is, those 
$\gamma\in\Phi^+$ such that $\bar\gamma\not\in \oE^\perp$, the corresponding 
homogeneous components in \eqref{eq:homogeneousGamma} are one-dimensional. Therefore, the 
type is $(128,56,0,0,0,0,0,1)$.

\smallskip

\item[$\dim\oE=2$] Here, if a positive root $\alpha\in\Phi^+$ satisfies $b_q(\bar\alpha,\oE)=0$, then
$\bar\alpha+\oE$ consists of $4$ nonisotropic vectors. Since $\oE^\perp/\oE$ is a regular quadratic
space with trivial Arf invariant of dimension $4$, it contains $6$ nonisotropic vectors, and thus
$\oE^\perp$ contains $6$ left cosets of the form $\bar\alpha+\oE$. This means that there are 
$12$ homogeneous spaces of dimension $4$ in \eqref{eq:homogeneousGamma}.

On the other hand, there are
$120-24=96$ positive roots $\alpha\in\Phi^+$ with $b_q(\bar\alpha,\oE)\neq 0$, and then the left coset
$\bar\alpha+\oE$ contains just two nonisotropic vectors, thus obtaining $48$ pairs of positive roots, which 
give $96$ homogeneous components in \eqref{eq:homogeneousGamma} of dimension $2$. Hence, the
type is $(0,96,0,12,0,0,0,1)$.

\smallskip

\item[$\dim\oE=3$] In this case, the quadratic space $\oE^\perp/\oE$ is a regular quadratic space
with trivial Arf invariant of dimension $2$, so it contains a unique nonisotropic element. Thus, 
if $\alpha\in\Phi^+$ is in $\oE^\perp$, the coset $\bar\alpha+\oE$ consists of the $8$ nonisotropic vectors
in $\oE^\perp$. This means that the homogeneous components $\cL_{(\alpha+E,\bar 0)}$, 
$\cL_{(\alpha+E,\bar 1)}$, and $\cL_{(E,\bar 1)}$ are eight-dimensional, and actually Cartan 
subalgebras, because the homogeneous components in a special grading on a semisimple Lie algebra are always
toral \cite[3.8 Theorem]{Hesselink}.

For any of the remaining $112$ positive roots $\alpha\in\Phi^+$ with 
$b_q(\bar\alpha,\oE)\neq 0$, the coset $\bar\alpha+\oE$ contains $4$ nonisotropic vectors. This forces all the
other homogeneous components in \eqref{eq:homogeneousGamma} to be $4$-dimensional, so 
that the type is $(0,0,0,56,0,0,0,3)$.

\smallskip

\item[$\dim\oE=4$] Here $\oE^\perp=\oE$, and all the cosets $\bar\alpha+\oE$ contain $8$ nonisotropic
vectors. Thus, all homogeneous components have dimension $8$ and the type is
$(0,0,0,0,0,0,0,31)$. \qedhere
\end{description}
\end{proof}

\begin{theorem}\label{th:E7}
Up to equivalence, there are exactly four special pure gradings on a simple Lie algebra of type $E_7$
over an algebraically closed field of characteristic $0$. 
Their universal  groups and types are given in the next table:
\begin{center}
\begin{tabular}{ccc}
\bigstrut & Universal group & Type\\
\hline\hline
\bigstrut $\Gamma_{E_7}^8$&$\ZZ_2^8$&$(126,0,0,0,0,0,1)$\\
\hline
\bigstrut  $\Gamma_{E_7}^7$&$\ZZ_2^7$&$(66,30,0,0,0,0,1)$\\
\hline
\bigstrut  $\Gamma_{E_7}^6$&$\ZZ_2^6$&$(0,48,2,6,0,0,1)$\\
\hline
\bigstrut  $\Gamma_{E_7}^5$&$\ZZ_2^5$&$(0,0,0,28,0,0,3)$\\
\hline
\end{tabular}
\end{center}
\end{theorem}
\begin{proof}
The quadratic form $q\colon\oR\rightarrow \FF_2$ in \eqref{eq:q} is regular, but odd-dimensional. 
The radical of $b_q$ is one-dimensional: $\rad b_q=\FF_2a$, with $q(a)=1$. Moreover,
the nonisotropic vectors of $\oR$ are precisely the $63$ vectors $\bar\alpha\bydef \alpha+2R$ for 
$\alpha\in\Phi^+$ together with $a$. Any subspace of $\oR$ containing strictly $\rad b_q$ contains
other nonisotropic vectors besides $a$, so it is not of the form $\oE=E/2R$ with $E$ the complementary
lattice of a special pure grading. Also, the lattice $F$ with $\oF=\rad b_q$ cannot be the 
complementary lattice of a special pure grading, graded by its universal group, because
for any two different positive roots $\alpha,\beta\in\Phi^+$, $\bar\alpha-\bar\beta\neq a$, as
otherwise we would have $1=q(\bar\beta)=q(\bar\alpha+a)=q(\bar\alpha)+q(a)=1+1=0$, and then
\eqref{eq:Ecirc} is not fulfilled. Therefore, the complementary lattice $E$ of a special pure grading
satisfies that $\oE$ is a totally isotropic subspace of $\oR$ and, in particular, it does not contain
$\rad b_q$.
 
As a consequence, we get 
the following possibilities, according to \eqref{eq:homogeneousGamma}:
\begin{description}
\item[$\dim\oE=0$] As for $E_8$, the universal group is isomorphic to $\ZZ_2^8$ and the type
is $(126,0,0,0,0,0,1)$.

\item[$\dim\oE=1$] Here the quadratic space $\oE^\perp/\oE$ is regular of dimension $5$, so it contains
$16$ nonisotropic elements, one of them being $a+\oE$. For each $\alpha\in\Phi^+$ with 
$\bar\alpha\in\oE^\perp$, the coset $\bar\alpha+\oE$ consists of two nonisotropic vectors, so either
$\bar\alpha+\oE=\{\bar\alpha,\bar\beta\}$ for another positive root $\beta$, or 
$\bar\alpha+\oE=\{\bar\alpha,a\}$, which happens only for $\bar\alpha=a+u$, with $\oE=\FF_2u$. Hence
there are $15$ pairs of positive roots that come together in \eqref{eq:homogeneousGamma}. The other
$35=63-30$ positive roots give rise to one-dimensional homogeneous components and the type is
$(66,30,0,0,0,0,1)$.

\smallskip

\item[$\dim\oE=2$] In this case the quadratic space $\oE^\perp/\oE$ is regular of dimension $3$, so
it contains $4$ nonisotropic elements, one of them being $a+\oE$. For each $\alpha\in\Phi^+$ with 
$\bar\alpha\in\oE^\perp$, the coset $\bar\alpha+\oE$ consists of four nonisotropic vectors, which come
from four positive roots unless $\bar\alpha+\oE=a+\oE$ that contains three nonisotropic vectors
obtained from positive roots. For the remaining $48=63-15$ positive roots $\alpha$ with 
$\bar\alpha\not\in \oE^\perp$, the coset $\bar\alpha+\oE$ contains two nonisotropic vectors. The type
is then $(0,48,2,6,0,0,1)$.

\smallskip

\item[$\dim\oE=3$] Here the quadratic space $\oE^\perp/\oE$ is regular of dimension $1$. The
nonisotropic vectors of $\oE^\perp$ form the left coset $a+\oE$, that contains eight nonisotropic vectors of $\oR$:
the element $a$ itself, and seven elements of the form $\bar\alpha$ with $\alpha\in\Phi^+$.  Pick one
of them, so that $a+\oE=\bar\alpha+\oE$, then
the homogeneous components $\cL_{(\alpha+E,\bar 0)}$, 
$\cL_{(\alpha+E,\bar 1)}$, and $\cL_{(E,\bar 1)}$ are seven-dimensional Cartan subalgebras.

For the remaining $56=63-7$ positive roots $\alpha$ with 
$\bar\alpha\not\in \oE^\perp$, the coset $\bar\alpha+\oE$ contains four nonisotropic vectors. The type
is then $(0,0,0,28,0,0,3)$. \qedhere
\end{description}
\end{proof}

\begin{theorem}\label{th:E6}
Up to equivalence, there are exactly three special pure gradings on a simple Lie algebra of type $E_6$
over an algebraically closed field of characteristic $0$. 
Their universal groups and types are given in the next table:
\begin{center}
\begin{tabular}{ccc}
\bigstrut & Universal group & Type\\
\hline\hline
\bigstrut  $\Gamma_{E_6}^7$&$\ZZ_2^7$&$(72,0,0,0,0,1)$\\
\hline
\bigstrut $\Gamma_{E_6}^6$&$\ZZ_2^6$&$(32,20,0,0,0,1)$\\
\hline
\bigstrut $\Gamma_{E_6}^5$&$\ZZ_2^5$&$(0,24,0,6,0,1)$\\
\hline
\end{tabular}
\end{center}
\end{theorem}
\begin{proof}
The quadratic form $q\colon \oR\rightarrow \FF_2$ in \eqref{eq:q} is regular with Arf invariant $1$, and the
nonisotropic vectors are precisely the $36$ vectors $\bar\alpha$ for $\alpha\in\Phi^+$. Therefore, 
because of Proposition~\ref{pr:Phi_cap_E},
the complementary lattice $E$ satisfies that $\oE=E/2R$ is a totally isotropic subspace of $\bar R$, and 
hence it has dimension at most $2$, as the Arf invariant is not trivial. The following possibilities appear:

\begin{description}
\item[$\dim\oE=0$] As for $E_8$ and $E_7$, the universal group is isomorphic to $\ZZ_2^7$ 
and the type is $(72,0,0,0,0,1)$.

\smallskip

\item[$\dim\oE=1$] Here the quadratic four-dimensional space $\oE^\perp/\oE$ is regular and its Arf
invariant is $1$, so it contains
$10$ nonisotropic elements, of the form $\bar\alpha+\oE$, with $\alpha\in\Phi^+$ 
and $\bar\alpha\in\oE^\perp$. Each such left coset contains two nonisotropic vectors.
 The remaining
$16=36-20$ positive roots give rise to one-dimensional homogeneous components and the type is
$(32,20,0,0,0,1)$.

\smallskip

\item[$\dim\oE=2$] In this case the quadratic two-dimensional space $\oE^\perp/\oE$ is regular and
 its Arf invariant is $1$, so it contains
$3$ nonisotropic elements, of the form $\bar\alpha+\oE$, with $\alpha\in\Phi^+$ 
and $\bar\alpha\in\oE^\perp$. Each such left coset contains four nonisotropic vectors.
 The remaining
$24=36-12$ positive roots $\alpha$ satisfy that the left coset $\bar\alpha+\oE$ contains two 
nonisotropic vectors, and hence give rise to two-dimensional homogeneous components. The type is
$(0,24,0,6,0,1)$. \qedhere
\end{description}
\end{proof}

\bigskip

%%%%%%%%%%%%%%%%%%%%%%%%%%%%%

\section{Weyl groups of the special pure gradings 
on simple Lie algebras of types 
\texorpdfstring{$E_6$}{E6}, \texorpdfstring{$E_7$}{E7}, 
or \texorpdfstring{$E_8$}{E8}}\label{se:WeylSpecialPureE}

This section is devoted to computing the Weyl group of each special pure grading on the 
simple Lie algebras of types $E_6$, $E_7$, $E_8$ 
and to describing its action on the corresponding universal group.

Let $\Gamma$ be a special pure grading on a simple Lie algebra $\cL$ of type $E_6$, $E_7$, or $E_8$, 
with universal group
$G$. Let $s\in G$ be an element such that $\cL_s$ contains a Cartan subalgebra $\cH$. 
As the first result of this section, we are going to describe the structure of the group 
$W_s(\Gamma)$ in Theorem~\ref{th:WsGamma}.

Let $\Phi$ be the
root system relative to $\cH$, let $R=\ZZ\Phi$ be the root lattice, and consider the quotient $\oR=R/2R$. 
Let $\pi\colon R\rightarrow \oG$ be the
projection in \eqref{eq:pi} and let $q\colon\oR\rightarrow \FF_2$ be
the quadratic form in \eqref{eq:q}. 
Let $E$ be the complementary lattice
of $\Gamma$ and $\oE=E/2R$. Denote by $\opi\colon\oR\rightarrow \oG$ the projection induced from $\pi$ 
(recall that $G$ is an elementary $2$-group) and
let $H$ be the subgroup of $G$ such that $\oH\bydef H/\langle s\rangle=\opi(\oE^\perp)$. 

The quadratic form $q$ restricts to a quadratic form  $q_{\oE}\colon\oE^\perp\rightarrow \FF_2$. 
As $q$ is regular and $\oE$ is totally isotropic (see Section \ref{se:SpecialpureE}), $\oE$ is the radical of  
$q_{\oE}$. 
Hence $q_{\oE}$ is transferred by means of $\opi$ to a regular quadratic form 
$q_{H,s}\colon\oH\rightarrow \FF_2$  (where we regard $\oH$
as a vector space over $\FF_2$).

There appears a natural flag $\cF$ of subgroups of $G$:
\[
\cF:\quad 1\leq \langle s\rangle \leq H\leq G,
\]
and we may consider the stabilizer of this flag in $\Aut(G)$, denoted by $\Stab_{\Aut(G)}(\cF)$.
Any element $\varphi\in \Stab_{\Aut(G)}(\cF)$ induces an automorphism of $\oH=H/\langle s\rangle$, 
denoted by $\varphi_{H,s}$.

\begin{theorem}\label{th:WsG_flag}
Let $\Gamma$ be a special pure grading on a simple Lie algebra $\cL$ of type $E_6$, $E_7$, or $E_8$,
over an algebraically closed field of characteristic $0$, with 
universal group
$G$. With the notations preceding the theorem, we have
\[
W_s(\Gamma)=\{\varphi\in\Stab_{\Aut(G)}(\cF)\mid \varphi_{H,s}\in\Ort(q_{H,s})\}.
\]
\end{theorem}

\begin{proof}
Since $\oE$ is totally isotropic, we can choose a totally isotropic subspace $\oE'$ of $\oR$ such that 
$\oE\cap\oE'=0$ and $b_q\vert_{\oE\oplus \oE'}$ is nondegenerate. 
To do this, it suffices to use induction on $\dim \oE$ and the fact that any isotropic vector is contained 
in a hyperbolic plane 
(see, e.g., the proof of \cite[Lemma 8.10]{EKM}).
Letting $\oW$ be the orthogonal complement to $\oE\oplus\oE'$ relative to $b_q$,
we get $\oR=\oE\oplus \oE'\oplus \oW$.

Identify $R/E\simeq \oR/\oE$ with $\oE'\oplus \oW$, and $G$ with $(\oE'\oplus \oW)\times \FF_2$ by
means of \eqref{eq:GRE2}. In this way, $H$ corresponds to $\oW\times \FF_2$.
Then we must prove that $W_s(\Gamma)$, seen as a group of
transformations of $(\oE'\oplus \oW)\times \FF_2$, consists of the linear maps
\begin{equation}\label{eq:WsGamma}
 \left(
 \begin{array}{c|c}
 \begin{matrix}f& 0\\ \alpha& g \end{matrix} &\begin{matrix} 0\\ 0\end{matrix}\\
\hline\\[-8pt]
\xi &1
\end{array}
\right),
\end{equation}
with $f\in\GL(\oE')$, $g\in\Ort(q\vert_{\oW})$,  $\alpha\in\Hom(\oE',\oW)$, and 
$\xi\in\Hom(\oE'\oplus\oW,\FF_2)$.

Because of Proposition~\ref{pr:CartanDieudonne}, the homomorphism 
$\rho'\colon\Stab_{\Aut(\Phi)}(\oE)\rightarrow \GL(R/E)$ in \eqref{eq:rho'} induces a homomorphism 
\[
\begin{split} \bar\rho\colon\Stab_{\Ort(q)}(\oE)&\longrightarrow \GL(R/E)\\
         \mu&\mapsto\  \bar\mu,
\end{split}
\]
where $\bar\mu(r+E)=\mu(r+2R)+\oE\in \oR/\oE\simeq R/E$ for any $r\in R$. Moreover, we have 
$\im\rho'=\im \bar\rho$.

Relative to the decomposition $\oR=\oE\oplus\oE'\oplus \oW$, the group 
$\Stab_{\Ort(q)}(\oE)$ consists of matrices of the form
\[
A=\begin{pmatrix}h & \beta & \gamma\\
  0&f&0\\
  0&\alpha&g \end{pmatrix}
\]
for $h\in \GL(\oE)$, $f\in \GL(\oE')$,  $g\in \Ort(q\vert_{\oW})$, and linear maps (over $\FF_2$) 
$\alpha\colon\oE'\rightarrow \oW$, $\beta\colon\oE'\rightarrow \oE$, $\gamma\colon \oW\rightarrow \oE$.
Then $A$ belongs to $\Ort(q)$ if and only if equations~\eqref{eq:fh}, \eqref{eq:alpha_gamma}, 
\eqref{eq:alpha_beta_f} hold.

Actually, we have:
\begin{itemize}
\item 
For any $u\in \oE$ and $v\in \oE'$,
\begin{equation}\label{eq:fh}
b_q\bigl(h(u),f(v)\bigr)=b_q\bigl(h(u),f(v)+\beta(v)\bigr)=b_q(u,v),
\end{equation}
so that any $f\in\GL(\oE')$ determines $h\in\GL(\oE)$ such that \eqref{eq:fh} holds, and 
conversely. 

\item 
For any $v\in\oE'$ and $w\in\oW$,
\[
\begin{split}
0&=b_q\bigl(\beta(v)+f(v)+\alpha(v),\gamma(w)+g(w)\bigr)\\
  &=b_q\bigl(f(v),\gamma(w)\bigr)+b_q\bigl(\alpha(v),g(w)\bigr),
\end{split}
\]
that is,
\begin{equation}\label{eq:alpha_gamma}
b_q\bigl(f(v),\gamma(w)\bigr)=b_q\bigl(\alpha(v),g(w)\bigr).
\end{equation}
Therefore, once $f$ and $g$ are fixed, any $\alpha$ determines $\gamma$ such that
\eqref{eq:alpha_gamma} holds. (And  if 
$\rad b_q=0$, then any $\gamma$ determines $\alpha$ too.)

\item
For any $v\in\oE'$,
\[
0=q(v)=q\bigl(\beta(v)+f(v)+\alpha(v)\bigr)=q\bigl(\alpha(v)\bigr)+b_q\bigl(f(v),\beta(v)\bigr),
\]
that is,
\begin{equation}\label{eq:alpha_beta_f}
q\bigl(\alpha(v)\bigr)=b_q\bigl(f(v),\beta(v)\bigr).
\end{equation}
For any $\alpha\in\Hom(\oE',\oW)$ and $f\in\GL(\oE')$, there is a linear map 
$\beta\colon\oE'\rightarrow\oE$ satisfying this equation.
\end{itemize}
The image under $\bar\rho$ of the matrix $A$ above is, under our identification 
$R/E\simeq \oE'\oplus \oW$,
the matrix
\[
 \begin{pmatrix}f& 0\\ \alpha&g \end{pmatrix}.
\]
For any $f\in\GL(\oE')$, $g\in\Ort(q\vert_{\oW})$, and $\alpha\in\Hom_{\FF_2}(\oE',\oW)$, there 
exist $h$, $\beta$, $\gamma$ satisfying equations~\eqref{eq:fh}, \eqref{eq:alpha_gamma}, 
\eqref{eq:alpha_beta_f},
and hence the result follows from Theorem~\ref{th:WsGamma}.
\end{proof}

\begin{corollary}[of the proof]
If $\cL$ in Theorem~\ref{th:WsG_flag} has type $E_6$ or $E_8$, then $W_s(\Gamma)$ is isomorphic to 
the affine orthogonal group of the restriction of $q$ to $\oE^{\perp}$.
\end{corollary}

\begin{proof}
Note that $\oE^\perp=\oE\oplus \oW$. Moreover, for $E_6$ and $E_8$, $\rad b_q$ is trivial, so the orthogonal group 
$\Ort(q\vert_{\oE^\perp})$ is isomorphic to the image of $\bar\rho$ by means of the map
\[
\begin{split}
\Theta\colon \Ort(q\vert_{\oE^\perp})&\longrightarrow \im\bar\rho\\
  \begin{pmatrix}h& \gamma\\ 0&g \end{pmatrix}&\mapsto\ \begin{pmatrix}f& 0\\ \alpha &g \end{pmatrix},
\end{split}
\]
for $h\in\GL(\oE)$, $g\in\Ort(q\vert_{\oW})$ and $\gamma\in\Hom(\oW,\oE)$, where $f\in\GL(\oE')$
is determined by \eqref{eq:fh} and $\alpha$ by \eqref{eq:alpha_gamma}.

On the other hand, $\oE^\perp$ is naturally isomorphic to $\Hom(\oR/\oE,\FF_2)$ by means of $b_q$. 
Denote this isomorphism by $\theta$. 

As we have seen, Theorem~\ref{th:WsGamma} shows that $W_s(\Gamma)$ is isomorphic to the semidirect product 
$\Hom(R/E,\FF_2)\rtimes \im\bar\rho\leq \Hom(R/E,\FF_2)\rtimes \GL(R/E)$, 
where the action of $\GL(R/E)$ on $\Hom(R/E,\FF_2)$ is 
the natural one: $f.\xi=\xi\circ f^{-1}$. We claim that our isomorphisms $\Theta$ and $\theta$ 
are compatible with the action, i.e., 
$\theta(Ax)=\Theta(A).\theta(x)$ for all $A\in \Ort(q\vert_{\oE^\perp})$ and $x\in \oE^\perp$.
Using the identification $R/E\simeq \oE'\oplus \oW$, this is equivalent to the equation 
$b_q(A(u+w),\Theta(A)(v'+w'))=b_q(u+w,v'+w')$ for all $u\in\oE$, $v'\in\oE'$, $w,w'\in\oW$, and 
$A=\left(\begin{smallmatrix}h& \gamma\\ 0&g\end{smallmatrix}\right)\in \Ort(q\vert_{\oE^\perp})$, 
which is easy to check:
\begin{equation*}%\label{eq:bqTheta}
\begin{split}
b_q\bigl(A(u+w),\Theta(A)(v'+w')\bigr)
  &=b_q\bigl(h(u)+\gamma(w)+g(w),f(v')+\alpha(v')+g(w')\bigr)\\
  &=b_q\bigl(h(u)+\gamma(w),f(v')\bigr)+b_q\bigl(g(w),\alpha(v')+g(w')\bigr)\\
  &=b_q(u,v')+b_q(w,w')=b_q(u+w,v'+w'),
\end{split}
\end{equation*}
where we have used \eqref{eq:fh} and \eqref{eq:alpha_gamma}. 

We conclude that $W_s(\Gamma)$ is isomorphic to the semidirect product 
$\oE^\perp\rtimes \Ort(q\vert_{\oE^\perp})$, 
which is the affine orthogonal group of the restriction of $q$ to $\oE^\perp$.
\end{proof}

%%%%%%%%%%%%%%%%%%%%

\begin{remark}\label{re:O(q_G)}
In case the complementary lattice $E$ is trivial: $E=2R$, which is
the case in Example~\ref{ex_finas}, one has $\oR=\oW$ in the proof of
Theorem~\ref{th:WsG_flag}, and hence $W_s(\Gamma)$ is the orthogonal group of the
degenerate quadratic form $q_G$ 
defined on $G$ (a $2$-elementary group, and hence a vector space over $\FF_2$), by means of
$q_G(g)\bydef q\bigl(\upsilon(g)\bigr)$, with $\upsilon\colon G\rightarrow \oR$ in \eqref{eq:upsilon}.
\end{remark}

\smallskip

According to Theorems~\ref{th:E8}, \ref{th:E7} and \ref{th:E6}, the special pure gradings on a simple
Lie algebra of type $E_6$, $E_7$, or $E_8$, contain a unique homogeneous component equal to a
 Cartan subalgebra, with
the exceptions of the special pure gradings on $E_8$ with universal groups isomorphic to 
$\ZZ_2^6$ and $\ZZ_2^5$, and the special pure grading on $E_7$ with universal  group 
isomorphic to $\ZZ_2^5$. 

The Weyl group of the special pure grading on $E_8$ with universal 
 group $\ZZ_2^5$ is the group of automorphisms of the universal  group, hence 
isomorphic to  $\GL_5(\FF_2)$ (see \cite[6.2 Theorem]{Hesselink}). This grading is an example
of Jordan grading (see \cite{Alek,Tho}). 

For the remaining cases, $W_s(\Gamma)=W(\Gamma)$ by
the uniqueness of the homogeneous component being equal to a Cartan subalgebra, so 
Theorem~\ref{th:WsG_flag} determines the whole Weyl group.

\smallskip

Therefore, we must deal with the special pure grading on $E_8$ with universal  group isomorphic
to $\ZZ_2^6$, and with the special pure grading on $E_7$ with universal  group isomorphic
to $\ZZ_2^5$. The proofs of Theorems~\ref{th:E8} and \ref{th:E7} show that, in both cases, there
are exactly three homogeneous components that are Cartan subalgebras. The subgroup spanned by
the degrees of these homogeneous components is isomorphic to $\ZZ_2^2$.

\begin{proposition}\label{pr:E_3Cartan}
Let $\cL$ be a simple Lie algebra of type $E_7$ or $E_8$
over an algebraically closed field of characteristic $0$. Let $\Gamma:\cL=\bigoplus_{g\in G}\cL_g$ be
the unique, up to equivalence, special pure grading with universal group $G$ and with exactly three
 homogeneous components: $\cL_s$, $\cL_{s'}$ and $\cL_{s''}$, being 
Cartan subalgebras. Let $W(\Gamma)$ be its Weyl group, and let $W'(\Gamma)$ be the pointwise 
stabilizer in $W(\Gamma)$ of $\{s,s',s''\}$:
\[
W'(\Gamma)=\{w\in W(\Gamma)\mid w(s)=s,\, w(s')=s',\, w(s'')=s''\}.
\] 
Then there is a short exact
sequence
\[
1\rightarrow W'(\Gamma)\hookrightarrow W(\Gamma)\rightarrow S_3\rightarrow 1,
\]
where $S_3$ denotes the symmetric group of degree $3$, and the homomorphism 
$W(\Gamma)\rightarrow S_3$ is  induced from the action of $W(\Gamma)$
on the elements $s,s',s''\in G$.
\end{proposition}
\begin{proof}
In the two cases we are dealing with, we can identify $G$ with $R/E\times \FF_2$ as in \eqref{eq:GRE2} 
with $\oE=E/2R$ a
totally isotropic subspace of $\oR=R/2R$. Fix a decomposition $\oR=\oE\oplus \oE'\oplus \oW$ as in
the proof of Theorem~\ref{th:WsG_flag} and identify $R/E$ with $\oE'\oplus\oW$. 
For $E_8$, $\oW$ is a two-dimensional space and the Arf invariant
of $q\vert_{\oW}$ is trivial, so there is a unique vector $a\in\oW$ with $q(a)=1$. For $E_7$, 
$\oW=\rad b_q=\FF_2 a$, with $q(a)=1$. The three homogeneous components that are Cartan
subalgebras are  $\cL_{(0,1)}$, $\cL_{(a,0)}$, and $\cL_{(a,1)}$ (see \eqref{eq:homogeneousGamma}).
Write $s=(0,1)$, $s'=(a,0)$ and $s''=s+s'=(a,1)$. Equation~\eqref{eq:WsGamma} shows that
$W_s(\Gamma)$, as a group of transformations of $(\oE'\oplus \oW)\times\FF_2$, consists of
the linear maps
\begin{equation}\label{eq:WsGamma2}
 \left(
 \begin{array}{c|c}
 \begin{matrix}f& 0\\ \alpha&g \end{matrix} &\begin{matrix} 0\\ 0\end{matrix}\\
\hline\\[-8pt]
\begin{matrix} \epsilon& \omega\end{matrix} &1
\end{array}
\right),
\end{equation}
with $g\in\Ort(q\vert_{\oW})$, $f\in\GL(\oE')$, $\alpha\in\Hom(\oE',\oW)$, 
$\epsilon\in\Hom(\oE',\FF_2)$, and $\omega\in\Hom(\oW,\FF_2)$. Note that $\Ort(q\vert_{\oW})$ is 
cyclic of order $2$ and fixes $a$ for $E_8$, while it is trivial for $E_7$. Therefore, $W'(\Gamma)$
consists of the linear maps in \eqref{eq:WsGamma2} with $\omega(a)=0$.

The linear maps of the form
\[
 \left(
 \begin{array}{c|c}
 \begin{matrix}\id& 0\\ 0&\id \end{matrix} &\begin{matrix} 0\\ 0\end{matrix}\\
\hline\\[-8pt]
\begin{matrix} 0&\omega\end{matrix} &1
\end{array}
\right),
\]
 with $\omega(a)=1$ have order $2$ and permute the labels $s'$ and $s''$, while fixing $s$.
 
However, the three Cartan subalgebras $\cL_s$, $\cL_{s'}$ and $\cL_{s''}$ play the same role, so we
could have started with the Cartan subalgebra $\cH=\cL_{s'}$ to conclude that there are elements
in $W(\Gamma)$ permuting $\cL_s$ and $\cL_{s''}$. We conclude that the natural homomorphism
$W(\Gamma)\rightarrow S_3$ is surjective, and its kernel is, by definition $W'(\Gamma)$.
\end{proof}

\begin{corollary}\label{co:E_3Cartan}
Let $\cL$ be a simple Lie algebra of type $E_7$ or $E_8$
over an algebraically closed field of characteristic $0$. Let $\Gamma:\cL=\bigoplus_{g\in G}\cL_g$ be
the unique, up to equivalence, special pure grading with universal group $G$ and with exactly 
three homogeneous components being 
Cartan subalgebras. Then the order of the Weyl group is given in the following table:
\begin{center}
\begin{tabular}{cc}
\bigstrut \quad Type of $\cL$\quad\null & $\left|W(\Gamma)\right|$\\
\hline\hline
\bigstrut$E_7$&$2^{10}\times 3^2\times 7$\\
\hline
\bigstrut$E_8$&$2^{15}\times 3^2\times 7$\\
\hline
\end{tabular}
\end{center}
\end{corollary}

\begin{proof}
By Proposition~\ref{pr:E_3Cartan}, the order of $W(\Gamma)$ is $3\times\left| W_s(\Gamma)\right|$.
But \eqref{eq:WsGamma} gives
\[
\begin{split}
\left| W_s(\Gamma)\right|&=\left|\Ort(q\vert_{\oW})\right|\times \left|\GL(\oE')\right|\times
\left|\Hom(\oE',\oW)\right|\times \left|\Hom(\oE'\oplus\oW,\FF_2)\right|\\
 &=\begin{cases}
 1\times(7\times 6\times 4)\times 2^3\times 2^4=2^{10}\times 3\times 7, &\text{for $E_7$,}\\
 2\times(7\times 6\times 4)\times 2^6\times 2^5=2^{15}\times 3\times 7, &\text{for $E_8$,}
 \end{cases}
\end{split}
\]
 as required.
\end{proof}

Our next goal is to show that the short exact sequence in 
Proposition~\ref{pr:E_3Cartan} splits, and to 
give a simpler description of the Weyl groups. To do this, let us consider nice realizations of these
gradings following ideas from \cite{CunhaElduque} and the references therein.

For $E_8$, let $H$ be the extended Hamming $[8,4,4]$ binary code, which is the
four dimensional subspace of $\FF_2^8$ that consists of the following words:
\[
\begin{split}
&\bcero=(0,0,0,0,0,0,0,0),\ \buno=(1,1,1,1,1,1,1,1), \\
&(1,1,1,1,0,0,0,0),\ (0,0,0,0,1,1,1,1),\\
&(1,1,0,0,1,1,0,0),\ (1,1,0,0,0,0,1,1),\ (0,0,1,1,1,1,0,0),\ (0,0,1,1,0,0,1,1),\\
&(1,0,1,0,1,0,1,0),\ (1,0,1,0,0,1,0,1),\ (0,1,0,1,1,0,1,0),\ (0,1,0,1,0,1,0,1),\\
&(1,0,0,1,1,0,0,1),\ (1,0,0,1,0,1,1,0),\ (0,1,1,0,1,0,0,1),\ (0,1,1,0,0,1,1,0).
\end{split}
\]
There is an $H$-grading $\underline\Gamma$ of the simple Lie algebra $\cL$ of
type $E_8$ with 
\[
\cL_{\bcero}=\frsl(V_1)\oplus\cdots\oplus \frsl(V_8),\qquad\cL_{\buno}=0,
\]
where $V_i$ is a two-dimensional vector space for any $i=1,\ldots,8$, and for any word 
$\bc\neq \bcero,\buno$, we have
\[
\cL_{\bc}=V_{c_1}\otimes V_{c_2}\otimes V_{c_3}\otimes V_{c_4}
\]
as a module for $\cL_{\bcero}$, with $1\leq c_1<c_2<c_3< c_4\leq 8$ being the slots of $\bc$
that contain $1$. For example, we have
\[
\cL_{(1,0,0,1,1,0,0,1)}=V_1\otimes V_4\otimes V_5\otimes V_8.
\]
The Lie bracket in $\cL$ is shown in \cite[Section 3.3]{EldMagic}.

There is a group homomorphism
\begin{equation}\label{eq:Psi}
\Psi\colon\SL(V_1)\times \cdots\times \SL(V_8)\longrightarrow \Aut(\cL)
\end{equation}
such that $\Psi\bigl((f_1,\ldots,f_8)\bigr)$ acts on $\frsl(V_i)\leq\cL_{\bcero}$ by $g\mapsto f_igf_i^{-1}$,
and on the component $V_{\bc}=V_{c_1}\otimes V_{c_2}\otimes V_{c_3}\otimes V_{c_4}$ as 
$f_{c_1}\otimes f_{c_2}\otimes f_{c_3}\otimes f_{c_4}$.

Fix bases on the $V_i$'s and consider the following automorphisms
\begin{equation}\label{eq:sigma_tau}
\begin{split}
\sigma&=\Psi\left(\Bigl(\begin{pmatrix} 0&1\\ -1&0\end{pmatrix},\cdots,
         \begin{pmatrix} 0&1\\ -1&0\end{pmatrix}\Bigr)\right),\\
 \tau&=\Psi\left(\Bigl(\begin{pmatrix} \bi&0\\ 0&-\bi\end{pmatrix},\cdots,
         \begin{pmatrix} \bi&0\\ 0&-\bi\end{pmatrix}\Bigr)\right),
\end{split}
\end{equation}
with $\bi^2=-1$.
From the conditions
\[
\begin{pmatrix} 0&1\\ -1&0\end{pmatrix}^2=-\id=\begin{pmatrix} \bi&0\\ 0&-\bi\end{pmatrix}^2,
\qquad
\begin{pmatrix} 0&1\\ -1&0\end{pmatrix}\begin{pmatrix} \bi&0\\ 0&-\bi\end{pmatrix}  
=- \begin{pmatrix} \bi&0\\ 0&-\bi\end{pmatrix} \begin{pmatrix} 0&1\\ -1&0\end{pmatrix},
\]
we get 
\[
\sigma^2=\tau^2=\id,\qquad \sigma\tau=\tau\sigma,
\]
because we have $\Psi\bigl((-\id,\ldots,-\id)\bigr)=\id$.

Consider the grading by $H\times\FF_2^2\simeq \ZZ_2^6$:
\begin{equation}\label{eq:E8_HF22}
\Gamma:\cL=\bigoplus_{(\bc,i,j)\in H\times\FF_2^2}\cL_{(\bc,i,j)}
\end{equation}
obtained by refining the previous $H$-grading $\underline\Gamma$ by means of $\sigma$ and $\tau$:
\[
\begin{split}
\cL_{(\bc,0,0)}&=\{x\in\cL_{\bc}\mid \sigma(x)=x=\tau(x)\},\\
\cL_{(\bc,1,0)}&=\{x\in\cL_{\bc}\mid \sigma(x)=-x=-\tau(x)\},\\
\cL_{(\bc,0,1)}&=\{x\in\cL_{\bc}\mid \sigma(x)=x=-\tau(x)\},\\
\cL_{(\bc,1,1)}&=\{x\in\cL_{\bc}\mid \sigma(x)=-x=\tau(x)\}.
\end{split}
\]

\begin{lemma}\label{le:HF22}
Up to equivalence, $\Gamma$ is the special pure grading on $E_8$ with universal  group
 $\ZZ_2^6$.
\end{lemma}
\begin{proof}
It is enough to note that $\frsl_2(V_i)$ is the direct sum of the three Cartan subalgebras spanned
by the linear maps with coordinate matrices in our fixed basis 
$\left(\begin{smallmatrix}0&1\\ -1&0\end{smallmatrix}\right)$,
$\left(\begin{smallmatrix}\bi&0\\ 0&-\bi\end{smallmatrix}\right)$, and
$\left(\begin{smallmatrix}0&\bi\\ \bi&0\end{smallmatrix}\right)$. It follows that $\cL_{(\bcero,0,0)}$ 
is trivial, so that $\Gamma$ is special, and $\cL_{(\bcero,1,0)}$, $\cL_{(\bcero,0,1)}$, and
$\cL_{(\bcero,1,1)}$ are Cartan subalgebras whose sum is 
$\cL_{\bcero}=\frsl(V_1)\oplus\cdots\oplus\frsl(V_8)$. In particular, $\Gamma$ is pure.

Moreover, for $\bc\neq \bcero,\buno$ in $H$, $\cL_{\bc}$ splits into the direct sum of four homogeneous
components of $\Gamma$, each of dimension $4$. Hence, $\Gamma$ has exactly three
homogeneous components that are Cartan subalgebras.
\end{proof}

Now, we are in a position to prove that the short exact sequence in Proposition~\ref{pr:E_3Cartan} splits
for type $E_8$,
and to provide a nice description of the Weyl group.

\begin{theorem}\label{th:E8_3Cartan}
Let $\cL$ be a simple Lie algebra of type $E_8$
over an algebraically closed field of characteristic $0$. Let $\Gamma:\cL=\bigoplus_{g\in G}\cL_g$ be
the unique, up to equivalence, special pure grading with universal group $G$ and with exactly three
 homogeneous components: $\cL_s$, $\cL_{s'}$ and $\cL_{s''}$, being 
Cartan subalgebras. Let $W(\Gamma)$ be its Weyl group, and let $W'(\Gamma)$ be the subgroup
$W'(\Gamma)=\{w\in W(\Gamma)\mid w(s)=s,\, w(s')=s',\, w(s'')=s''\}$. Then the short exact
sequence
\[
1\rightarrow W'(\Gamma)\hookrightarrow W(\Gamma)\rightarrow S_3\rightarrow 1,
\]
in Proposition~\ref{pr:E_3Cartan} splits. 

Moreover, consider the flag
\[
\cF:\quad 1\leq S\leq K\leq G,
\]
where $S=\langle s,s',s''\rangle=\langle g\in G\mid \cL_g\ \text{is a Cartan subalgebra}\rangle$, and 
$K=\langle g\in G\mid \cL_g=0\rangle$. Then the Weyl group is the stabilizer of $\cF$:
\[
W(\Gamma)=\Stab_{\Aut(G)}(\cF).
\]
\end{theorem}

\begin{proof}
We may identify $G$ with $H\times\FF_2^2$ as in
\eqref{eq:E8_HF22}. Then the subgroups $S$ and $K$ equal $\FF_2^2$ and 
$\langle\buno\rangle\times \FF_2^2$ respectively, because the only elements $g$ in $G$ with 
$\cL_g=0$ are $(\buno,i,j)$, for $i,j\in\FF_2^2$.

The elements of the Weyl group clearly preserve the subgroups $S$ and $K$. Therefore, 
the Weyl group, considered as a
subgroup of $\Aut(H\times\FF_2^2)$, consists of transformations that stabilize the flag $\cF$ or, in other
words, transformations of the form:
\begin{equation}\label{eq:Weyl_E8_3Cartan}
\left(
 \begin{array}{c|c}
f &0\\
\hline\\[-10pt]
\mu &g
\end{array}
\right), 
\end{equation}
with $f\in\Stab_{\GL(H)}(\buno)$, $g\in\GL_2(2)$, and $\mu\in \Hom(H,\FF_2^2)$.

But an easy computation gives:
\[
\begin{split}
\left|\Stab_{\GL(H)}(\buno)\right|&=2^3\times\left| \GL_3(2)\right|=2^3\times 7\times 6\times 4
             =2^6\times 3\times 7,\\
\left|\Hom(H,\FF_2^2) \right|&=2^8,\\
\left|\GL_2(2) \right|&=3\times 2,
\end{split}
\]
so the order of the subgroup in \eqref{eq:Weyl_E8_3Cartan} is 
$2^{15}\times 3^2\times 7$, which coincides with
the order of $W(\Gamma)$ by Corollary~\ref{co:E_3Cartan}. Therefore, $W(\Gamma)$ is the group
of all transformations in \eqref{eq:Weyl_E8_3Cartan}.

Moreover, $W'(\Gamma)$ is the subgroup of $W(\Gamma)$ consisting of the transformations that fix
$(\bcero,i,j)$ for any $i,j\in\FF_2^2$, and hence of the transformations in \eqref{eq:Weyl_E8_3Cartan}
with $g=\id$, which is complemented by the subgroup of transformations in \eqref{eq:Weyl_E8_3Cartan}
with $f=\id$, $\mu=0$. This shows that the short exact sequence in Proposition~\ref{pr:E_3Cartan}
splits.
\end{proof}

For $E_7$ the situation is simpler. Let $\cL$ be the simple Lie algebra of type $E_7$, and now let
$C$ be the simplex $[7,4,3]$ binary linear code, which is the three dimensional subspace
of  $\FF_2^7$ consisting of the words: 
\[
\begin{split}
&\bcero=(0,0,0,0,0,0,0),\ (1,1,0,0,1,1,0), \\
&(0,1,1,0,0,1,1),\ (1,0,1,0,1,0,1),\\
&(1,1,1,1,0,0,0),\ (0,0,1,1,1,1,0),\\
&(1,0,0,1,0,1,1),\ (0,1,0,1,1,0,1).
\end{split}
\]
There is a grading $\underline\Gamma$ by  $C\simeq\ZZ_2^3$ with
\[
\cL_{\bcero}=\frsl(V_1)\oplus\cdots\oplus\frsl(V_7),
\]
and
\[
\cL_{\bc}=V_{c_1}\otimes V_{c_2}\otimes V_{c_3}\otimes V_{c_4}
\]
as for $E_8$. Again, this is refined to a grading $\Gamma$ by $C\times\FF_2^2\simeq \ZZ_2^5$ using
the natural group homomorphism 
$\Psi\colon\SL(V_1)\times \cdots\times \SL(V_7)\longrightarrow \Aut(\cL)$ 
as in
\eqref{eq:Psi} and the automorphisms $\sigma,\tau$ as in \eqref{eq:sigma_tau}, but with only seven
components. The same arguments as for $E_8$ prove the following result.

\begin{lemma}\label{le:CF22}
Up to equivalence, $\Gamma$ is the special pure grading on $E_7$ with universal  group
 $\ZZ_2^5$.
\end{lemma}

As in the case of $E_8$, we get that the short exact sequence in 
Proposition~\ref{pr:E_3Cartan} splits, and the Weyl group gets a nice description.

\begin{theorem}\label{th:E7_3Cartan}
Let $\cL$ be a simple Lie algebra of type $E_7$
over an algebraically closed field of characteristic $0$. Let $\Gamma:\cL=\bigoplus_{g\in G}\cL_g$ be
the unique, up to equivalence, special pure grading with universal group $G$ and with exactly three
 homogeneous components: $\cL_s$, $\cL_{s'}$ and $\cL_{s''}$, being 
Cartan subalgebras. Let $W(\Gamma)$ be its Weyl group,
and let $W'(\Gamma)$ be the pointwise stabilizer in $W(\Gamma)$ of 
$\{s,s',s''\}$:
\[
W'(\Gamma)=\{w\in W(\Gamma)\mid w(s)=s,\, w(s')=s',\, w(s'')=s''\}. 
\]
Then the short exact
sequence
\[
1\rightarrow W'(\Gamma)\hookrightarrow W(\Gamma)\rightarrow S_3\rightarrow 1,
\]
in Proposition~\ref{pr:E_3Cartan} splits. 

Moreover, consider the flag
\[
\cF:\quad 1\leq S\leq G,
\]
where $S=\langle s,s',s''\rangle=\langle g\in G\mid \cL_g\ \text{is a Cartan subalgebra}\rangle$. Then
the Weyl group is the stabilizer of $\cF$:
\[
W(\Gamma)=\Stab_{\Aut(G)}(\cF).
\]
\end{theorem}

\begin{proof}
Identify $G$ with $C\times\FF_2^2$. Then the subgroup $S$ equals $\FF_2^2$, and is preserved
by  $W(\Gamma)$,  so $W(\Gamma)$ preserves the flag $\cF$, and hence
consists of transformations of the form
\begin{equation}\label{eq:Weyl_E7_3Cartan}
\left(
 \begin{array}{c|c}
f &0\\
\hline\\[-10pt]
\mu &g
\end{array}
\right), 
\end{equation}
with $f\in\GL(C)$, $g\in\GL_2(2)$, and $\mu\in \Hom(C,\FF_2^2)$. Again, easy computations give:
\[
\begin{split}
\left|\GL(C)\right|&=\left|\GL_3(2)\right|= 7\times 6\times 4=2^3\times 3\times 7,\\
\left|\Hom(C,\FF_2^2) \right|&=2^6,\\
\left|\GL_2(2) \right|&=3\times 2,
\end{split}
\]
so the order of the subgroup in \eqref{eq:Weyl_E7_3Cartan} is $2^{10}\times 3^2\times 7$, 
which coincides with  the order of $W(\Gamma)$ by 
Corollary~\ref{co:E_3Cartan}. 
Therefore, $W(\Gamma)$ is the group of all transformations in \eqref{eq:Weyl_E7_3Cartan}, and the result follows.
\end{proof}

\begin{remark}
By duality, the quadratic form $q_{H,s}$ of Theorem~\ref{th:WsG_flag} corresponds to the regular 
quadratic form induced by the restriction of $\hat{q}$ of Remark~\ref{re:involution_types} to the 
subgroup $\Diag(\Gamma)\cap\cT_2$ of $\cT_2$. As a result, the Weyl groups $W(\Gamma)$ for 
$\cL$ of types $E_6$ and $E_8$ could be computed using, respectively, Propositions 6.9 and 8.14 in \cite{Yu}.
\end{remark}

\bigskip

%%%%%%%%%%%%%%%%%%%%%%%%%%%%%%%%%%%%%%

\section{Classification of special pure gradings up to isomorphism}\label{se:isomorfismo}

Recall that any special grading 
$\Gamma: \cL=\bigoplus_{g\in G}\cL_g$ on a semisimple Lie algebra
over an algebraically closed field of characteristic $0$ is almost fine
(\cite[Proposition 3.8]{EK_almostfine}). If $U$ is the universal 
group of $\Gamma$, and we denote by $\Delta$  the grading $\Gamma$ when considered as a grading
over $U$, then there is a group homomorphism $\alpha\colon U\rightarrow G$ such that $\Gamma$ is the
coarsening ${}^\alpha\!\Delta$. This homomorphism $\alpha$ is bijective on the supports.

If now $\Gamma': \cL=\bigoplus_{g\in G}\cL'_g$ is another special $G$-grading equivalent to $\Gamma$,
then there is a group homomorphism $\beta\colon U\rightarrow G$, with $\Gamma'={}^\beta\!\Delta$. 
Then the gradings $\Gamma$ and $\Gamma'$ are isomorphic if and only if there is an element
$\omega$ in the Weyl group $W(\Delta)$ such that $\alpha=\beta\circ\omega$ 
(\cite[Theorem 4.3]{EK_almostfine}).

\smallskip

The goal of this section is the classification of the special pure gradings on the simple Lie algebras
of types $E_6$, $E_7$, and $E_8$, up to isomorphism.

This will be achieved by defining suitable invariants for these gradings 
(Definition~\ref{df:inv}).

\smallskip

Let $\Gamma: \cL=\bigoplus_{g\in G}\cL_g$ be a special pure grading on the simple Lie algebra $\cL$, of
type $E_6$, $E_7$, or $E_8$, over an algebraically closed field of characteristic $0$, but now we 
do not assume that $G$
is the universal group of the grading. Pick an element $s$ in the support of $\Gamma$ such that
$\cL_s$ is a Cartan subalgebra. Denote by $T$ the subgroup of $G$ generated by the support, which is
a $2$-group, as it is a quotient of the universal group, and pick a character $\chi$ of $T$ with 
$\chi(s)=-1$. Then $T=\ker\chi\times\langle s\rangle$.
Write $\oT=T/\langle s\rangle$, which is isomorphic to $\ker\chi$.

As in \eqref{eq:pi}, there is a surjective group homomorphism
\[
\pi\colon R\rightarrow \oT.
\]
The \emph{complementary lattice} $E\bydef \ker\pi$ contains $2R$. The grading $\Gamma$ being special,
we have $E\cap\Phi=\emptyset$ as in Proposition~\ref{pr:Phi_cap_E}. As in \eqref{eq:GRE2} we conclude
that $T$ is isomorphic to $R/E\times\FF_2$.

Write, as usual, $\oR=R/2R$ and
$\oE=E/2R$, and consider the induced linear map $\opi\colon \oR\rightarrow \oT$, and the 
quadratic form $q\colon\oR\rightarrow \FF_2$ in \eqref{eq:q}.

\begin{lemma}\label{le:E7exception}
Unless $\cL$ is of type $E_7$ and $\Gamma$ is equivalent to the special pure grading 
$\Gamma_{E_7}^8$ (Theorem~\ref{th:E7}), the subgroup $T$ is the universal group of $\Gamma$.
\end{lemma}
\begin{proof}
From $E\cap\Phi=\emptyset$ we conclude, using the arguments previous to 
Proposition~\ref{pr:CartanDieudonne}, that either $\oE$ is a totally isotropic subspace of $\oR$, or
$\cL$ is of type $E_7$ and $\oE=\rad b_q$.

In the first case, $T$ is the universal group, as so is $R/E\times\FF_2$. In the second case,
$\Gamma$ is equivalent to $\Gamma_{E_7}^8$ in Theorem~\ref{th:E7}, and its
 universal group is $\oR\times\FF_2$, while $T$ is isomorphic to its quotient 
$\oR/\rad(b_q)\times\FF_2$.
\end{proof}

\begin{definition}\label{df:inv}
Let $\Gamma: \cL=\bigoplus_{g\in G}\cL_g$ be a special pure grading on the simple Lie algebra $\cL$ of
type $E_6$, $E_7$, or $E_8$, over an algebraically closed field of characteristic $0$. Define the sequence 
$\Inv(\Gamma)$
as follows:
\setlength{\leftmargini}{25pt}
\begin{romanenumerate}
\item 
If $\Gamma$ is equivalent to $\Gamma_{E_8}^5$ in Theorem~\ref{th:E8}, then 
\[
\Inv(\Gamma)=(T),
\]
with $T$ the subgroup of $G$ generated by the support of $\Gamma$.

\item
If $\Gamma$ is equivalent to $\Gamma_{E_8}^6$ in Theorem~\ref{th:E8}, then 
\[
\Inv(\Gamma)=(T,K,S),
\]
with $T=\langle\supp\Gamma\rangle$ as before, 
$S=\langle t\in T\mid \cL_t\ \text{is a Cartan subalgebra of $\cL$}\rangle$, and
$K=\langle t\in T\mid \cL_t=0\rangle$.

\item
If $\Gamma$ is equivalent to $\Gamma_{E_7}^5$ in Theorem~\ref{th:E7}, then
\[
\Inv(\Gamma)=(T,S),
\]
with $T=\langle\supp\Gamma\rangle$ and 
$S=\langle t\in T\mid \cL_t\ \text{is a Cartan subalgebra of $\cL$}\rangle$.

\item
If $\Gamma$ is equivalent to $\Gamma_{E_7}^8$ in Theorem~\ref{th:E7}, but the subgroup
$T=\langle\supp\Gamma\rangle$ is not the universal group (in other words, $T$ is isomorphic to
$\ZZ_2^7$), then
\[
\Inv(\Gamma)=(T,S,b_{T,S}),
\]
with $T=\langle\supp\Gamma\rangle$,
 $S=\langle t\in T\mid \cL_t\ \text{is a Cartan subalgebra of $\cL$}\rangle\,(\simeq \ZZ_2)$, 
$\oT=T/S$,
and $b_{T,S}\colon \oT\times\oT\rightarrow \FF_2$ the nondegenerate alternating bilinear form
defined as follows: the polar form $b_q$ of the regular quadratic form $q\colon\oR\rightarrow \FF_2$
induces a nondegenerate alternating bilinear form 
$\bar b_{q}\colon \oR/\rad(b_q)\times \oR/\rad(b_q)\rightarrow \FF_2$, and $b_{T,S}$ is the induced
form on $\oT$:
\[
b_{T,S}\bigl(\opi(x),\opi(y)\bigr)\bydef \bar b_q(x+\rad b_q,y+\rad b_q)= b_q(x,y),
\]
for any $x,y\in\oR$.

\item
In the remaining cases, so that $\Gamma$ is equivalent to either $\Gamma_{E_8}^r$, $r=7,8,9$,
in Theorem~\ref{th:E8}, or
$\Gamma_{E_7}^r$, $r=6,7$, in Theorem~\ref{th:E7}, or $\Gamma_{E_7}^8$ with 
$\langle\supp\Gamma\rangle\simeq \ZZ_2^8$, or $\Gamma_{E_6}^r$, $r=5,6,7$, in Theorem~\ref{th:E6},
then
\[
\Inv(\Gamma)=(T,H,S,q_{H,S}),
\]
with $T=\langle\supp\Gamma\rangle$, 
$S=\langle t\in T\mid \cL_t\ \text{is a Cartan subalgebra of $\cL$}\rangle$ $(\simeq \ZZ_2)$, $\oT=T/S$,
$H$ the subgroup of $T$ such that $\oH=H/S=\opi(\oE^\perp)$, where $E$ is the complementary lattice,
and $q_{H,S}\colon \oH\rightarrow \FF_2$ the regular quadratic form transferred from 
$q\vert_{\oE^\perp}$ by means of $\opi$:
$
q_{H,S}\bigl(\opi(x)\bigr)=q(x)
$,
for any $x\in \oE^\perp$.
\end{romanenumerate}
\end{definition}

The classification of special pure gradings up to isomorphism is now easy:

\begin{theorem}\label{th:iso}
Let $\Gamma: \cL=\bigoplus_{g\in G}\cL_g$ and $\Gamma': \cL=\bigoplus_{g\in G}\cL'_g$ be two 
special pure gradings on the simple Lie algebra $\cL$ of
type $E_6$, $E_7$, or $E_8$, over an algebraically closed field of characteristic $0$. 
Then $\Gamma$ is isomorphic to 
$\Gamma'$ if and only if $\Inv(\Gamma)=\Inv(\Gamma')$.
\end{theorem}
\begin{proof} 
It is clear that if $\Gamma$ and $\Gamma'$ are isomorphic, then $\Inv(\Gamma)=\Inv(\Gamma')$.

For the converse, 
let us start with a simple case. Assume that $\Gamma$ and $\Gamma'$ are equivalent to 
$\Gamma_{E_8}^6$ and $\Inv(\Gamma)=\Inv(\Gamma')=(T,K,S)$. As in Theorem~\ref{th:E8_3Cartan}, the
universal group of $\Gamma_{E_8}^6$ is $U=H\times\FF_2^2$, and let
$\alpha,\beta\colon U\rightarrow G$ be the corresponding group homomorphisms
so that $\Gamma={}^\alpha(\Gamma_{E_8}^6)$ and 
$\Gamma'={}^\beta(\Gamma_{E_8}^6)$. Then $\alpha$ and $\beta$ are 
isomorphisms when considered as maps 
$U\rightarrow T=\langle\supp\Gamma\rangle=\langle\supp\Gamma'\rangle$.
Then the composition $\beta^{-1}\circ\alpha\colon U\rightarrow U$ makes sense and it is an 
automorphism that preserves
the flag in Theorem~\ref{th:E8_3Cartan}. Therefore, $\beta^{-1}\circ\alpha$ belongs to the Weyl group
$W(\Gamma_{E_8}^6)$. It follows that $\alpha=\beta\circ\omega$ for an
element $\omega\in W(\Gamma_{E_8}^6)$, proving that $\Gamma$ and $\Gamma'$ are isomorphic.

\smallskip

Consider now the most problematic case, in which $\Gamma$ and $\Gamma'$ are equivalent to 
$\Gamma_{E_7}^8$, with $\Inv(\Gamma)=\Inv(\Gamma')=(T,S, b_{T,S})$ and
$T\simeq \ZZ_2^7$. The universal group of $\Gamma_{E_7}^8$ is $\oR\times\FF_2$ and let 
$\alpha,\beta\colon\oR\times\FF_2\rightarrow G$ be the corresponding group homomorphisms. These induce
surjective  group homomorphisms $\overline{\alpha},\overline{\beta}\colon \oR\rightarrow \oT=T/S$.

Take a complementary space $\oW$ of $\rad(b_q)$: $\oR=\rad(b_q)\oplus\oW$. Then 
$\overline{\alpha}$ and $\overline{\beta}$ give isomorphisms $\oW\rightarrow \oT$ and for any
$w_1,w_2\in \oW$, 
$b_q(w_1,w_2)=b_{T,S}\bigl(\overline{\alpha}(w_1),\overline{\alpha}(w_2)\bigr)
=b_{T,S}\bigl(\overline{\beta}(w_1),\overline{\beta}(w_2)\bigr)$, so that 
$\overline{\beta}^{-1}\overline{\alpha}$ lies in the symplectic group of $b_q\vert_{\oW}$.

But if $V$ is an odd dimensional vector space over $\FF_2$ endowed with a 
regular quadratic form $Q$, $W$ is a complementary subspace of $\rad(b_Q)$ and
 $A\in\SP(W,b_Q\vert_W)$, then there is a unique element $B\in\Ort(V,Q)$, such that 
$B(w)-A(w)\in\rad(b_Q)$ for any $w\in W$. (Actually, we have $\SP(W,b_Q\vert_W)\simeq\Ort(V,Q)$.)
Indeed, it is enough to define, with $\rad(b_Q)=\FF_2 a$, $B(a)=a$ and $B(w)=A(w)+\lambda(w)a$,
for a linear form $\lambda$. Then $B$ is an isometry if and only if 
$Q(w)=Q\bigl(B(w)\bigr)=Q\bigl(A(w)+\lambda(w)a\bigr)=Q\bigl(A(w)\bigr)+\lambda(w)^2=
Q\bigl(A(w)\bigr)+\lambda(w)$ for any $w\in W$, if and only if $\lambda(w)=Q\bigl(A(w)\bigr)+q(w)$
for any $w\in W$ (this is linear because $A\in\SP(W,b_Q\vert_W)$!).

Hence there is  $\varphi\in\Ort(\oR,q)$ such that 
$\varphi(x)+\overline{\beta}^{-1}\overline{\alpha}(x)\in \rad(b_q)$, and then $\alpha=\beta\circ\omega$
for an element $\omega\in W(\Gamma_{E_7}^8)$ of the form 
\[
\left(\begin{array}{c|c} \varphi &0\\
\hline\\[-10pt] \epsilon&1\end{array}\right)
\] 
as in \eqref{eq:WsGamma2}, and thus $\Gamma$ and $\Gamma'$ are isomorphic. 

\smallskip

The remaining cases are easy because, as for $\Gamma_{E_8}^6$, the homomorphisms 
$\alpha,\beta$ are bijections from the
universal group of the almost fine grading to the invariant $T$ of $\Gamma$ and $\Gamma'$.
\end{proof}

%-------------------------------------------

\end{document}